\documentclass[preprint]{elsarticle}
\usepackage{lineno}
\modulolinenumbers[5]

\journal{J. Comput. Appl. Math.}

\usepackage{amsmath}
\usepackage{amssymb}
\usepackage{amsthm}
\usepackage{graphicx}
\usepackage{epic,eepic,epsfig}
\usepackage{color}
\usepackage{subfigure}  
\usepackage{placeins}   
\usepackage{multirow}
\usepackage{epstopdf}
\usepackage{varwidth}
\usepackage{float}
\usepackage[table]{xcolor}

\newtheorem{theorem}{Theorem}[section]
\newtheorem{example}{Example}[section]
\newtheorem{lemma}{Lemma}[section]
\newtheorem{remark}{Remark}[section]

\newtheorem{algorithm}{Algorithm}[section]

\definecolor{tabclr}{cmyk}{0,0,1,0}

\newcommand{\bx}{\boldsymbol{x}}

\newcommand{\balpha}{\boldsymbol{\alpha}}

\newcommand{\Rmnum}[1]{\uppercase\expandafter{\romannumeral #1}} 

\begin{document}

\title{Adaptive-Multilevel BDDC algorithm for three-dimensional plane wave Helmholtz systems}

\author[SCNU]{Jie Peng\corref{cof}}
\ead{pengjie18@m.scnu.edu.cn}

\author[XTU1,XTU2]{Shi Shu\corref{cof}}
\ead{shushi@xtu.edu.cn}

\author[XTU1,XTU2]{Junxian Wang\corref{cor}}
\ead{wangjunxian@xtu.edu.cn}

\author[SCNU]{Liuqiang Zhong}
\ead{zhong@scnu.edu.cn}

\cortext[cof]{These authors contributed equally to this work and should be considered co-first authors}
\cortext[cor]{Corresponding author}
\address[SCNU]{School of Mathematical Sciences, South China Normal University, Guangzhou 510631, China}
\address[XTU1]{School of Mathematics and Computational Science, Xiangtan University, Xiangtan 411105, China}
\address[XTU2]{Hunan Key Laboratory for Computation and Simulation in Science and Engineering,
Xiangtan University, Xiangtan 411105, China}

\begin{abstract}
In this paper, we are concerned with the weighted plane wave  least-squares (PWLS) method for three-dimensional Helmholtz equations, and develop the multi-level adaptive BDDC algorithms for solving the resulting discrete system.
In order to form the adaptive coarse components, the local generalized eigenvalue problems for each common
face and each common edge are carefully designed.
The condition number of the two-level adaptive BDDC preconditioned system is proved to be bounded above by a
user-defined tolerance and a constant which is dependent on the maximum number of faces and edges per subdomain and the number of subdomains sharing a common edge.
The efficiency of these algorithms is illustrated on a benchmark problem.
The numerical results show the robustness of our two-level adaptive BDDC algorithms with respect to the wave number, the number of subdomains and the mesh size,
and illustrate that our multi-level adaptive BDDC algorithm can reduce the scale of the coarse problem and can be used to solve large wave number problems efficiently.
\end{abstract}


\maketitle

\section{Introduction}

The plane wave method is an important class of Trefftz method \cite{HiptmairMoiola16:237} for solving the Helmholtz equations with high wave numbers. Compared with the traditional finite element methods (FEMs), the plane wave methods attract people's attention mainly for two reasons:
(i) they need fewer degrees of freedom (dofs) with the same precision requirements;
(ii) their basis functions naturally satisfy the unconstrained homogenous Helmholtz equations without considering boundary conditions.
In addition, compared with boundary element methods (BEMs), an advantage of the plane wave methods is that they do not require the evaluation of singular integrals. Examples of this method include ultra weak variational formulation (UWVF)\cite{CessenatDespres98:255,GamalloAstley07:406},
 variational theory of complex rays (VTCR)\cite{KovalevskyLadeveze12:142,LadevezeArnaud01:193}, weighted
plane wave least squares (PWLS) method \cite{MonkWang99:121,HuYuan14:587,HuYuan18:245}, plane wave discontinuous Galerkin method (PWDG) \cite{GittelsonHiptmair09:297,HiptmairMoiola11:264} and so on.

The variational problem of the PWLS method is derived from a quadratic functional minimization problem.
Compared with other plane wave methods, the advantage is that the coefficient matrix of the corresponding discrete system is Hermitian and  positive definite and can be solved by the preconditioned conjugate gradient (PCG) method.
As early as 1999, Monk and Wang first proposed the PWLS method for solving Helmholtz equations \cite{MonkWang99:121}.
After that, inspired by VTCR method, Hu and Yuan proposed a weighted PWLS method and established the corresponding error estimation theory.
But it should be pointed out that this work is only applied to the case of homogeneous equation.
Recently, Hu and Yuan extended the PWLS method to the non-homogeneous case by combining local spectral element method in \cite{HuYuan18:245}.
Although the PWLS discrete system of the Helmholtz equation with high wave numbers is Hermitian positive definite, its coefficient matrix is still highly ill-conditioned \cite{HiptmairMoiola16:237}. Specifically, the condition number of the coefficient matrix will deteriorate sharply with the decrease of the mesh size or the increase of the number of plane wave basis functions in each element.
In addition, compared with the traditional FEMs, the plane wave basis function is defined on the element rather than on the nodes or edges, which results in failure of the existing fast algorithms applied to such systems directly. Therefore, it is very challenging to design a preconditioned algorithm for it.

Domain decomposition method (DDM) is a popular method to construct efficient preconditioners \cite{ToselliWidlund05Book}.
With the development of modern computer parallel architecture, this method has become a powerful tool for numerical simulation of complex practical problems. How to construct efficient parallel DDM has become one of the hot research fields of the current scientific computing.
The traditional DDMs for solving the Helmholtz equations with high wave numbers face the problem that the well-posedness of local problem can not be guaranteed, which results in that the traditional algorithms for positive definite problems, such as the traditional Schwarz method, Neumann-Neumann method, Finite Element Tearing and Interconnecting (FETI) and Dual-Primal Finite Element Tearing and Interconnecting (FETI-DP) method, can not be directly applied to solve this kind of indefinite problem.

In order to overcome this difficulty, some improved methods have been proposed. For example,
Farhat et al. proposed the FETI-H and FETI-DPH methods by using regularization technique and combining with the coarse space constructed by plane wave functions \cite{FarhatMacedo99:231,FarhatAvery05:499}; Gander, Magoules and Nataf proposed the optimal Schwarz algorithm by improving the transmission boundary conditions and the selection of the optimal parameters \cite{GanderMagoules02:38}, and based on this work, Gander, Halpern and Magoules also proposed
an optimized Schwarz method with two-sided Robin transmission conditions in \cite{GanderHalpern07:163};
Chen, Liu and Xu proposed a kind of two-parameter relaxed Robin DDM by choosing appropriate Robin parameters and relaxation parameters in \cite{ChenLiu16:921}; In addition, inspired by the sweeping preconditioner (an approximation preconditioner) with optimal computational complexity which is proposed by Engquist and Ying  \cite{EngquistYing11:686}, Chen and Xiang designed a source transfer DDM \cite{ChenXia13:538}.

BDDC (Balancing Domain Decomposition by Constraints) method which was first proposed by Dohrmann for structural mechanics problems \cite{Dohrmann03:246} is an important non-overlapping DDM.
Based on the principle of energy minimization by constraints, Mandel first derived the convergence theory of the BDDC method in \cite{MandelDohrmann03:639}, and proved that the condition number of its preconditioned system is $C(1 + log^2(H/h))$, where $h$ and $H$ represent the size of mesh and subdomain respectively.
Later, this method has been widely used to solve various PDE(s) models, such as scalar diffusion problem \cite{MandelDohrmann03:639},
linear elasticity problem \cite{GippertKlawonn12:2208}, almost incompressible elasticity problem \cite{PavarinoWidlund10:3604},
Helmholtz problem \cite{LiTu09:745,TuLi09:75}, Stokes flow problem \cite{SistekSousedik11:429}, porous media flow problem \cite{TuXM07:146,ZampiniTu17:A1389},
isogeometric analysis \cite{DaPavarino14:A1118},  etc..
In particular, a robust BDDC method is designed by adding plane wave continuity constraints for solving the FEM discrete system of the Helmholtz equation with constant wave numbers in \cite{LiTu09:745,TuLi09:75}.
However, when the PDE(s) model contains strongly discontinuous or highly oscillating coefficients,
the BDDC method with the standard coarse space may not converge any more \cite{KimChung15:571}.
Therefore, it is particularly important to select coarse space adaptively according to the characteristics of the problem.

The adaptive BDDC method is an advanced BDDC method, for which the primal unknowns are always selected by solving some local generalized eigenvalue problems \cite{MandelSousedik07:1389}.
Since the condition number of the corresponding preconditioned system is controlled by a user defined tolerence,
it has attracted extensive attention of many scholars and has been successfully extended to FEMs \cite{KimChung15:571,KimChung18:64,KlawonnRadtke16:301,OhWidlund18:659,ZampiniVassilevski17:103}, mortar methods \cite{PengShu18:185}, staggered discontinuous Galerkin methods \cite{KimChung17:599}
and isogeometric analysis \cite{DaVeigaPavarino17:A281} and so on.
However, most of the available literatures
on adaptive BDDC algorithms aimed at some real symmetric positive definite systems and were presented in algebraic form (i.e. matrices and vectors),
and adaptive BDDC
algorithms for PWLS discrete systems (Hermitian positive definite and highly ill-conditioned) of three-dimensional Helmholtz equation with high wave numbers in variational form have not previously been discussed in the literature.
It is worth pointing out that the variational form is as popular as the matrix representation for describing and analyzing BDDC and adaptive BDDC methods \cite{BrennerSung07:1429,KlawonnRadtke16:301}.
Therefore, how to construct scalable and efficient adaptive BDDC preconditioners in variational form for such discrete systems and establish relevant theories is still a work worthy of further study.

Based on our earlier work on adaptive BDDC algorithms in variational form for the PWLS discretization of the Helmholtz problem in two-dimension \cite{PengWang18:683}, we extend these algorithms to three-dimensional problems
and establish the corresponding condition number estimation theory.
For the three-dimensional PWLS Helmholtz system studied in this paper,
the local generalized eigenvalue problems are formed for each common face and each common edge respectively.
To be more specific, the common face is an equivalence class shared by two subdomains and thus the generalized eigenvalue problem is identical to that considered for two-dimensional problems in \cite{PengWang18:683}, and the common edge is an equivalence class shared by more than two subdomains and thus a different idea is required to form an appropriate generalized eigenvalue problem.
Inspired by the estimate of condition numbers of the preconditioned matrix, the local generalized eigenvalue problem on each common edge is designed carefully.
In addition, though the condition numbers can be controlled by a user-defined tolerance, the cost for forming the generalized eigenvalue problems is quite considerable \cite{Zampini16:S282}, especially for three-dimensional problems. Thus similar to \cite{KlawonnRadtke16:75}, we use economic-version to enhance the efficiency of the proposed method.
Further, we extend the applicability of these methods to the case of high wave numbers with a specific focus on the multilevel extension.
Since the number of primal unknowns increases as the wave number or the number of subdomains increases,
we attempt to construct a multilevel adaptive BDDC algorithm to resolve the bottleneck in solving large-scale coarse
problem.
Finally, we perform numerical experiments for a benchmark problem.
These results verify the correctness of theoretical
results and show the efficiency of our two-level adaptive BDDC algorithms with respect to the angular frequency, the
number of subdomains, and mesh size. And the numerical results also show that the multi-level adaptive BDDC algorithm is effective for reducing the number of dofs in the coarse problem,
and can be used to solving large wave number problems efficiently.

This paper is organized as follows. In Section 2, a brief introduction to the PWLS method for three-dimensional Helmholtz equations is presented.
In Section 3, a two-level BDDC preconditioner with adaptive coarse space is
proposed, and then the multilevel extension of these methods is carried out.
The condition number analysis  is provided in Section 4 and various numerical experiments are presented in Section 5.
Conclusions will be given in section 6.

\section{Problem formulation}\label{sec:2}
\setcounter{equation}{0}

In this section, we briefly review the weighted plane wave least squares formulation for the Helmholtz equation.

\subsection{Model problem}

Let $\Omega \in \mathbb{R}^3$ be a bounded and connected Lipschitz domain,
the boundary of $\Omega$ is given as
$$\partial \Omega = \overline{\Gamma_d \cup \Gamma_n \cup \Gamma_r},$$
where $\Gamma_d, \Gamma_n$ and $\Gamma_r$ are disjoint sets.
Consider the following Helmholtz equation with general boundary condition
\begin{equation}\label{model equation}
\left\{
\begin{array}{rcll}
 -\Delta u  - \kappa^2 u &=&  0        &in~ \Omega,\\
      u &=& g_d & on ~\Gamma_d,\\
      \partial_{\bf n} u &=& g_n & on~\Gamma_n,\\
     (\partial_{\bf n} + i \kappa)u &=&  g_r    &    on~ \Gamma_r,
\end{array}
\right.
\end{equation}
where $i = \sqrt{-1}$ is the imaginary unit, the operator $\partial_{\bf n}$ is the outer normal derivative,
and $\kappa = \omega/c > 0$ is the wave number. Here $\omega$ is called the angular frequency and $c$ is the wave speed.

\subsection{Weighted plane wave least squares discretization}

Let $\Omega$ be divided into a partition as follows
$$
\bar{\Omega} = \bigcup\limits_{k=1}^{N_h} \bar{\Omega}_k,
$$
where the hexahedron elements $\{\Omega_k\}$ satisfy that $\Omega_m \cap \Omega_l = \emptyset, m \neq l$, $h_k$ is the size of $\Omega_k$ and $h = \max\limits_{1\le k \le N_h} h_k$.
Define
\begin{equation*}
\begin{array}{l}
\gamma_{kj} = \partial \Omega_k \cap \partial \Omega_j,~~ \mbox{for}~k,j = 1,\cdots,N_h ~\mbox{and}~ k\neq j, \\
\gamma_k = \partial \Omega_k \cap \partial \Omega,~~\mbox{for}~k=1,\cdots,N_h,
\end{array}
\end{equation*}
and let
\begin{equation}\label{def-FB-FI}
\mathcal{F}_B = \bigcup\limits_{k=1}^{N_h} \gamma_k,~~\mathcal{F}_I = \bigcup\limits_{k \neq j} \gamma_{kj}.
\end{equation}

In this paper, we assume that each
$\kappa_k:= \kappa|_{\Omega_k}$ is a constant. $V(\Omega_k)$ is denoted as the space of the functions which satisfies the
homogeneous Helmholtz's equation \eqref{model equation} on the cavity $\Omega_k$:
\begin{equation*}
V(\Omega_k) = \{v_k \in H^1({\Omega_k}):~\Delta v_k + \kappa_k^2 v_k = 0\},~k=1,\cdots,N_h.
\end{equation*}
Define
$$
V({\mathcal{T}_h}) = \bigcup_{k=1}^{N_h} V(\Omega_k).
$$

The problem \eqref{model equation} to be solved consists in finding $u_k: = u|_{\Omega_k} \in \{v \in H^1({\Omega_k}):~\nabla v \in H(div;\Omega_k)\}$
such that
\begin{equation}\label{submodel}
\left\{
\begin{array}{rcll}
    -\Delta u_k  -\kappa_k^2 u_k &=&  0    & in ~\Omega_k, \\
      u_k &=& g_d & on ~\partial \Omega_k \cap \Gamma_d,\\
      \partial_{\bf n} u_k &=& g_n & on~\partial \Omega_k \cap  \Gamma_n,\\
     (\partial_{\bf n} + i \kappa)u_k &=&  g_r    &    on~ \partial \Omega_k \cap  \Gamma_r,
\end{array}
\right. k=1,2,\cdots,N_h,
\end{equation}
and
\begin{equation}\label{interface condition}
\left\{
\begin{array}{rcll}
                              u_k  - u_j   &=&  0    &   over ~\gamma_{kj}\\
     \partial_{{\bf n}_k}u_k +  \partial_{{\bf n}_j}u_j&=&  0    &   over ~\gamma_{kj}
\end{array}
\right. ~~~k,j = 1,\cdots,N_h~and~k\neq j.
\end{equation}

The variational problem associated with the plane wave least squared  approximation of problem \eqref{submodel} and \eqref{interface condition} can be expressed as: find $u\in V(\mathcal{T}_h)$ such that
\begin{eqnarray}\label{115-variational-form}
   a(u,v)=  \mathcal{L}(v),~\forall v\in V({\mathcal T_h}),
\end{eqnarray}
where
\begin{align}\nonumber
  a(u,v) &= \sum_{k=1}^{N_h} \left(\theta_{k1} \int_{\gamma_k \cap \Gamma_d} u_k \cdot \overline{v_k} ds
          + \theta_{k2} \int_{\gamma_k \cap \Gamma_n} \partial_{{\bf n}_k} u_k \cdot \overline{\partial_{{\bf n}_k} v_k} ds \right.\\\nonumber
         &~~~~~~~~\left.+ \theta_{k3} \int_{\gamma_k \cap \Gamma_r} ((\partial_{{\bf n}_k}+i\kappa_k )u_k) \cdot \overline{(\partial_{{\bf n}_k}+i \kappa_k)v_k}ds\right)\\\nonumber
         &+ \sum_{j\neq k} \left(\alpha_{kj}\int_{\gamma_{kj}} (u_k-u_j)\cdot\overline{(v_k-v_j)}ds \right.\\\label{115-1}
         &~~~~~~~~\left.+ \beta_{kj}\int_{\gamma_{kj}} ({\partial_{{\bf n}_k}u_k+\partial_{{\bf n}_j}u_j})\cdot\overline{(\partial_{{\bf n}_k}v_k+\partial_{{\bf n}_j}v_j)}ds\right),~~\forall u,v \in V(\mathcal{T}_h)\\\nonumber
\mathcal{L}(v) &= \sum_{k=1}^{N_h} \left(\theta_{k1} \int_{\gamma_k \cap \Gamma_d} g_d \cdot \overline{v_k} ds
                + \theta_{k2} \int_{\gamma_k \cap \Gamma_n} g_n \cdot \overline{\partial_{{\bf n}_k} v_k} ds \right.\\\label{115-2}
               &~~~~~~~~\left.+ \theta_{k3} \int_{\gamma_k \cap \Gamma_r} g_r \cdot \overline{(\partial_{{\bf n}_k}+i \kappa_k)v_k}ds\right),~~\forall v \in V(\mathcal{T}_h),
\end{align}
here $\overline{\diamond}$ denotes the complex conjugate of the complex quantity $\diamond$, the Lagrange multipliers
$$\alpha_{kj} = h^{-1} + |\kappa_{kj}|,~\beta_{kj} = h^{-1} |\kappa_{kj}|^{-2} + |\kappa_{kj}|^{-1}~ \mbox{with}~\kappa_{kj} = (\kappa_k + \kappa_j)/2,$$
and
$$\theta_{k1} = h^{-1} + |\kappa_k|,~ \theta_{k2} = \theta_{k3} = h^{-1} |\kappa_k|^{-2} + |\kappa_k|^{-1}.$$

It is clear that $a(\cdot,\cdot)$ is sesquilinear, and similar to  the proof of Theorem 3.1 in \cite{HuYuan14:587}, we can see
that
$a(\cdot,\cdot)$ is Hermitian positive definite on $V(\mathcal{T}_h)$.

\subsection{Discretization of the variational formulation}

In this subsection, we derive a discretization of the variational formulation \eqref{115-variational-form}.

Let $p \ge 1$ be a given positive integer, $y_{m,l}(l=1,\cdots,p)$ be the wave shape functions on $\Omega_m (m=1,\cdots,N_h)$, which satisfy
\begin{equation*}
\left\{
\begin{array}{rcl}
 y_{m,l}(\bx) &=& e^{i\kappa(\bx \cdot\balpha_l)},~\bx \in \bar{\Omega}_m,\\
 |\balpha_l| &=& 1, \\
   \balpha_l &\neq& \balpha_s,~\mbox{for}~l \neq s,
\end{array}
\right.
\end{equation*}
where $\balpha_l$ $(l = 1,\cdots,p)$ are unit wave propagation directions.
In particular, during numerical simulations, we set
\begin{equation*}
\balpha_l := \balpha_{r,j} = \left(
  \begin{array}{c}
    \cos(2\pi(r-1)/n_1) \cos(\pi(j-1)/n_2)\\
    \cos(2\pi(r-1)/n_1) \sin(\pi(j-1)/n_2)\\
    \sin(2\pi(r-1)/n_1)\\
  \end{array}
\right),~r=1,\cdots,n_1, j=1,\cdots,n_2,
\end{equation*}
where $l = (j-1)n_1 + r$, $n_1, n_2$ are two positive integers, and we choose $n_1$ according to the rules in \cite{HuLi17:1242} that when $n_2$ is odd, we set $n_1 = 2n_2 - 1$ or $2n_2$ or $2n_2+1$; when $n_2$ is even, we set $n_1 = 2n_2-1$ or $2n_2+1$.

Thus we can define a finite dimensional subspace of  $V({\mathcal T_h})$ as
\begin{equation*}
V_p({\mathcal T_h}) =\mathrm{span}\{\varphi_{m,l}:~1 \le l \le p, 1\le m \le N_h\},
\end{equation*}
where
\begin{equation*}
\varphi_{m,l}(\bx)=
\left\{
\begin{array}{ll}
y_{m,l}(\bx)     &\bx \in \bar{\Omega}_m,\\
     0          &\bx \in \Omega \backslash \bar{\Omega}_m.
\end{array}
\right.
\end{equation*}

For ease of notations, we denote $\{\varphi_{m,l}\}$ briefly by $\{\varphi_s\}$, where $s = (m-1)p + l$.
Define $\mathcal{S}_h = \{1,\cdots,dim(V_p(\mathcal{T}_h))\}$ as the number set of dofs.

Let $V_p(\mathcal{T}_h)$ be the plane wave finite dimensional space defined above. Then the discrete variational
problem associated with \eqref{115-variational-form} can be described as follows: find $u \in V_p({\mathcal T_h})$ such that
\begin{eqnarray}\label{115-dis}
a(u,v) =  \mathcal{L}(v),~\forall v \in V_p({\mathcal T_h}),
\end{eqnarray}
where $a(\cdot,\cdot)$ and $\mathcal{L}(\cdot)$  are separately defined in \eqref{115-1} and \eqref{115-2}.

Note that the above system is large and highly ill-conditioned when the
wave number is large,
therefore, a fast solver for \eqref{115-dis} will be discussed in the rest of this paper.

\section{Adaptive BDDC algorithms}\label{sec:3}
\setcounter{equation}{0}

\subsection{Globs}

{\bf Globs} (or {\bf Equivalence classes})  \cite{PechsteinDohrmann17:273} of  all dofs play a very important role in design, analysis and parallel implementation of the BDDC methods.
An important step in designing a non-overlapping domain decomposition method is to classify all dofs.
Different from the discretizations which dofs are defined on the vertices or edges of the mesh,
the dofs in the PWLS discretization are defined on the elements;
therefore, to classify all the dofs,  we need to introduce a non-overlapping domain decomposition and special interface which is similar to the fat interface in \cite{DaPavarino14:A1118}.

Let $\mathcal{T}_d = \{D_r\}_{r=1}^{N_d}$ be a non-overlapping subdomain partition of $\Omega$, where each $D_r$ consists of several complete elements and part
of the elements in $\mathcal{T}_h$ (Fig. \ref{chap4-fig-dd-1} shows the cross-section of the elements and subdomains).

\begin{figure}[H]
  \centering
  \includegraphics[width=0.3\textwidth]{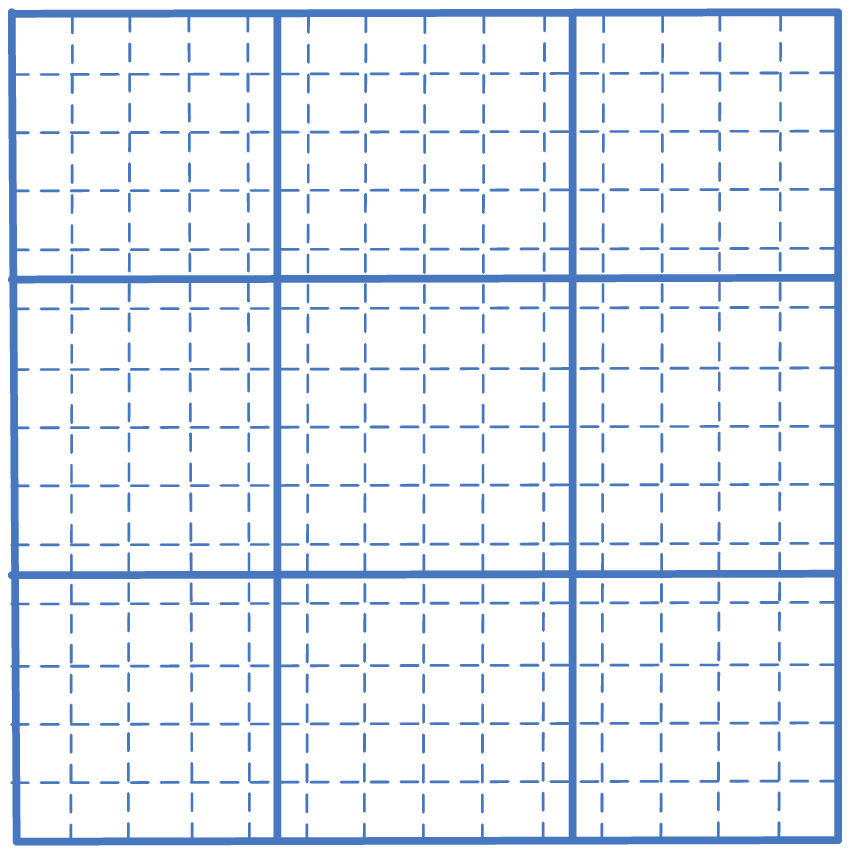}\\
  \caption{The cross-section of the elements:  the small square with dotted line boundary; the cross-section of the subdomains: the square with solid line boundary.}\label{chap4-fig-dd-1}
\end{figure}

Denote $\mathcal{S}_d := \{1,\cdots,N_d\}$. For each global dof indice $k \in \mathcal{S}_h$, we define the set of subdomain indices sharing the dof $k$ as
\begin{eqnarray*}
\mathcal{N}_k = \{r \in \mathcal{S}_d:~ \mu_k^{(r)} = 1\},
\end{eqnarray*}
where
\allowdisplaybreaks
\begin{eqnarray*}
\mu_k^{(r)} = \left\{
\begin{array}{ll}
1, & \mathrm{supp}(\varphi_k) \cap D_r \neq \emptyset,\\ 
0, & \mbox{others},
\end{array}
\right.
\end{eqnarray*}
here~$\mathrm{supp}(\varphi_k)$ denotes the support of the basis function $\varphi_k$.

Using~$\{\mathcal{N}_k\}$, we can decompose $\mathcal{S}_h$ into some globs.
Specifically, the global dof indices $k$ and $k'$ belong to the same glob when $\mathcal{N}_k = \mathcal{N}_{k'} (k\neq k')$.
Denote~$\mathcal{G}$ as the set of all globs, and $\mathcal{N}_G$ as the set of subdomain indices associated with glob $G$, then the globs associated with $D_r$ can be expressed as
\begin{eqnarray*}
\mathcal{G}_r :=\{G \in \mathcal{G}:~r\in \mathcal{N}_G\}, ~r = 1,\cdots,N_d.
\end{eqnarray*}

Let $|\diamond|$ denote the size of the set $\diamond$.
If~$|\mathcal{N}_G| = 1$ and $\mathcal{N}_{G}:=\{r\}$, we call~$G$ the set of global dof indices in $D_r$ (specified as $\mathcal{I}_r$), and denote the set of global dof indices in the interior of all subdomains by $\mathcal{I}:=\cup_{r=1}^{N_d}\mathcal{I}_r$.
If~$|\mathcal{N}_G| = 2$ and $\mathcal{N}_{G}:=\{r,j\}$, we call~$G$ the set of global dof indices in the common face of $D_r$ and $D_j$ (specified as $\mathcal{F}_k$, $n_{F_k} := |\mathcal{F}_k|$, and $k$ is the common face indice), and denote the set of global dof indices in all the common faces by $\mathcal{F}:= \cup_{k=1}^{N_F}\mathcal{F}_k$, where $N_F$ is the number of the common faces.
If~$|\mathcal{N}_G| = 4$ and $\mathcal{N}_{G}:=\{r,l,m,n\}$, we call~$G$ the set of global dof indices in the common edge of $D_r$, $D_l$, $D_m$ and $D_n$ (specified as $\mathcal{E}_k$, $n_{E_k} := |\mathcal{E}_k|$, and $k$ is the common edge indice), and denote the set of global dof indices in all the common edges by $\mathcal{E}:= \cup_{k=1}^{N_E}\mathcal{E}_k$, where $N_E$ is the number of the common edges.
Further, if~$|\mathcal{N}_G| > 4$, then we call~$G$ the set of global dof indices in the common vertex (specified as $\mathcal{V}_k$, $n_{V_k} := |\mathcal{V}_k|$, and $k$ is the common vertex indice),
and denote the set of global dof indices in all the common vertices by~$\mathcal{V}:=\cup_{k=1}^{N_V}\mathcal{V}_k$, where~$N_V$ is the number of the common vertices.
In addition, let the set of global dof indices on the interface be~$\mathcal{F} \cup \mathcal{E} \cup \mathcal{V}$.

From the above description, we can see that $\mathcal{N}_{\mathcal{F}_k}$, $\mathcal{N}_{\mathcal{E}_k}$ and~$\mathcal{N}_{\mathcal{V}_k}$ denote the set of the subdomain indices sharing by the $k$-th common face, the $k$-th common edge and the $k$-th common vertex, respectively.
For simplicity, we denote the $k$-th ($k=1,\cdots,N_F$) common face as $F_k$, the $k$-th $(k=1,\cdots,N_E)$ common edge as $E_k$, and the $k$-th $(k=1,\cdots,N_V)$ common vertex as $V_k$.

For each subdomain $D_r$, define
\begin{align}\label{chap2-def-M-i}
\mathcal{M}_{X}^{(r)}   &= \{k:~\mathcal{X}_k \subset \mathcal{G}_r,~\mbox{for}~1\le k \le N_{X}\}, ~X = F, E, V.
\end{align}
It must be pointed out that the notation $\mathcal{X}_k (\mathcal{X}=\mathcal{F}, \mathcal{E}, \mathcal{V})$ is the glob which is associated with $X_k (X=F, E, V)$ in \eqref{chap2-def-M-i} and the remainder of this paper.
From this definition, we can see that $\mathcal{M}_F^{(r)} (\mathcal{M}_E^{(r)}, \mathcal{M}_V^{(r)})$ denotes the set of common face (edge, vertex) indices associated with subdomain $D_r$. See Fig. \ref{chap4-fig-dd-1-1} for a two-dimensional example, where $\mathcal{M}_{F}^{(1)} = \{1,7\}$, $\mathcal{M}_{F}^{(5)} = \{3,4,8,11\}$, $\mathcal{M}_{V}^{(1)} = \{1\}$, $\mathcal{M}_{V}^{(5)} = \{1,2,3,4\}$.
\begin{figure}[H]
  \centering
  \includegraphics[width=0.33\textwidth]{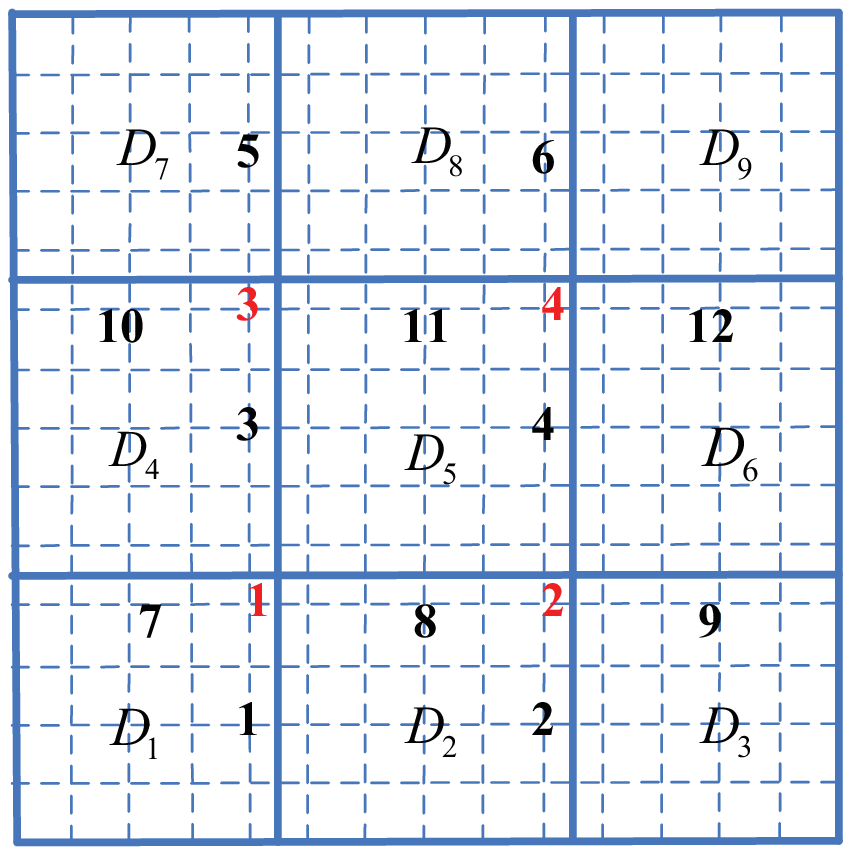}\\
  \caption{The global indices of the common faces and common vertexes for a 2D example.}\label{chap4-fig-dd-1-1}
\end{figure}

The following descriptions are based on the above notations.

\subsection{Schur complement problem}

In this subsection, the Schur complement system of the discrete variational problem \eqref{115-dis} with respect to the dofs on the interface are introduced.
For this purpose, we need to introduce some spaces firstly.

Using the bases $\{\varphi_l\}$ of~$V_p(\mathcal{T}_h)$, we can define some subspaces of $V_p(\mathcal{T}_h)$ as
\begin{align*}
& V_I = \oplus_{r=1}^{N_d}V^{(r)}_I, ~V^{(r)}_I = \mathrm{span}\{\varphi_l:~l \in \mathcal{I}_r\},~r = 1,\cdots,N_d,\\
& V_{X_k} = \mathrm{span}\{\varphi_{l}:~l \in \mathcal{X}_{k}\},~k=1,\cdots,N_X,~X = F, E, V.
\end{align*}

Moreover, some local function spaces in each subdomain $D_r$ are introduced as
\begin{align*}
V_{X_k}^{(r)} = \mathrm{span}\{\varphi_{l}^{(r)}:~l \in \mathcal{X}_k\},~X = F, E, V,
\end{align*}
where $\{\varphi_l^{(r)}\}$ are the truncated functions of the basis $\{\varphi_l\}$
whose support is completely contained in $D_r$, i.e.
\begin{align}\label{chap2-def-trunc-func}
\varphi_l^{(r)}(\bx) = \left\{\begin{array}{ll}
\varphi_l(\bx) & \bx \in \bar{D}_r\\
0 & \bx \in \Omega\backslash \bar{D}_r
\end{array}\right.,~r = 1,\cdots,N_d.
\end{align}
In what follows, we will denote $\zeta^{(r)}$ as the truncated function of $\zeta$ in $D_r$ when non confusion can arise.

Let
\begin{align*}
V^{(r)} = V^{(r)}_I \oplus (\oplus_{k\in \mathcal{M}_F^{(r)}} V_{F_k}^{(r)}) \oplus (\oplus_{k\in \mathcal{M}_E^{(r)}} V_{E_k}^{(r)}) \oplus (\oplus_{k\in \mathcal{M}_V^{(r)}} V_{V_k}^{(r)}),~r=1,\cdots,N_d,
\end{align*}
and for any $r = 1,\cdots, N_d$, define
\begin{align*}
\mathcal{F}_{I}^{(r)} = \{\tilde{\gamma}_{kj}:~ \tilde{\gamma}_{kj} = \gamma_{kj}|_{\bar{D}_r},~ \forall \gamma_{kj} \in \mathcal{F}_I\},~
\mathcal{F}_{B}^{(r)} = \{\tilde{\gamma}_k:~ \tilde{\gamma}_k = \gamma_{k}|_{\bar{D}_r},~ \forall \gamma_{k} \in \mathcal{F}_B\},
\end{align*}
where $\mathcal{F}_I$ and $\mathcal{F}_B$ are defined in \eqref{def-FB-FI}.

We can introduce a sesquilinear form $a_r(\cdot,\cdot)$ by
\begin{align}\nonumber
a_r(u,v) &= \sum_{\tilde{\gamma}_{kj} \in \mathcal{F}_I^{(r)}} \alpha_{kj}\int_{\tilde{\gamma}_{kj}} (u_k-u_j) \cdot \overline{(v_k-v_j)}ds
  \\ \nonumber
&~~~ + \sum_{\tilde{\gamma}_{kj} \in \mathcal{F}_I^{(r)}} \beta_{kj}\int_{\tilde{\gamma}_{kj}}
({\partial_{{\bf n}_k}u_k+\partial_{{\bf n}_j}u_j}) \cdot \overline{(\partial_{{\bf n}_k}v_k+\partial_{{\bf n}_j}v_j)}ds
  \\\nonumber
&~~~+ \sum_{\tilde{\gamma}_k \in \mathcal{F}_B^{(r)}} \left(\theta_{k1} \int_{\tilde{\gamma}_k \cap \Gamma_d} u_k \cdot \overline{v_k} ds
          + \theta_{k2} \int_{\tilde{\gamma}_k \cap \Gamma_n} \partial_{{\bf n}_k} u_k \cdot \overline{\partial_{{\bf n}_k} v_k} ds \right.\\\label{chap4-def-AUV}
         &~~~~~~~~~~~~~~~\left.+ \theta_{k3} \int_{\tilde{\gamma}_k \cap \Gamma_r} ((\partial_{{\bf n}_k}+i\kappa_k )u_k) \cdot \overline{(\partial_{{\bf n}_k}+i \kappa_k)v_k}ds\right),~u,v \in V^{(r)}.
\end{align}
Noting that~$\alpha_{kj}, \beta_{kj}, \theta_{k1}, \theta_{k2}, \theta_{k3} > 0$, it is easy to verify that $a_r(\cdot,\cdot)$ is Hermitian and
semi-positive definite in $V^{(r)}$. Therefore, we can define a semi-norm
\begin{align}\label{def-normAi}
|\cdot|^2_{a_r} := a_r(\cdot, \cdot).
\end{align}

Combining \eqref{115-1} with \eqref{chap4-def-AUV}, we know that $a_r (\cdot,\cdot)$ satisfies
\begin{eqnarray}\label{chap4-115-1-0}
  a(u,v) = \sum\limits_{r=1}^{N_d} a_r (u^{(r)},v^{(r)}),~\forall u,v \in V_p(\mathcal{T}_h),
\end{eqnarray}
where $u^{(r)}, v^{(r)}$ are the corresponding truncated functions of~$u,v$ in~$D_r$, and from \eqref{chap4-def-AUV}, \eqref{chap4-115-1-0} and the support property of these functions, it is easy to know that
\begin{align*}
a_r(u^{(r)}, v^{(r)}) = a(u^{(r)}, v^{(r)}),~\forall~u^{(r)}, v^{(r)} \in V_I^{(r)} \subset V_p(\mathcal{T}_h),~~r = 1,\cdots,N_d.
\end{align*}
Therefore, since $a(\cdot,\cdot)$ is Hermitian positive definite in $V_p(\mathcal{T}_h)$ and $V_I^{(r)} \subset V_p(\mathcal{T}_h)$, we know that $a_r(\cdot,\cdot)$ is Hermitian positive definite in $V_I^{(r)}$.

Based on the above preparation, we can introduce an important space of discrete harmonic functions which is
directly related to the Schur complement.

For any $\mathcal{X}_k (X = F, E, V)$, define
\begin{align}\label{chap2-def-phi-k-l-n}
\phi^{X_k}_{l} = \varphi_{k_l} + \sum\limits_{\nu \in \mathcal{N}_{\mathcal{X}_k}}\check{\phi}^{X_k,\nu}_{l},~l = 1,\cdots,n_{X_k},
\end{align}
where $\check{\phi}^{X_k,\nu}_{l} \in V_I^{(\nu)} (\nu \in \mathcal{N}_{\mathcal{X}_k})$ satisfies the following orthogonality
\begin{align}\label{chap2-def-Hw-i}
a_{\nu}(\check{\phi}^{X_k,\nu}_{l}, v) = -a_{\nu}(\varphi_{k_l}^{(\nu)}, v), ~~\forall v \in V_I^{(\nu)},
\end{align}
here~$k_l$ denotes the global number of the $l$-th dof in $\mathcal{X}_k$, $\varphi_{k_l}^{(\nu)}$ is the truncated function of the basis $\varphi_{k_l}$ in $D_{\nu}$, and $a_{\nu}(\cdot,\cdot)$ is defined in \eqref{chap4-def-AUV}.

Using the above functions, the function space associated with the Schur complement system of the discrete variational problem \eqref{115-dis} can be defined as
\begin{align}\label{chap2-def-W}
\hat{W} = (\oplus_{k=1}^{N_F} W_{F_k}) \oplus (\oplus_{k=1}^{N_E} W_{E_k}) \oplus (\oplus_{k=1}^{N_V} W_{V_k}),
\end{align}
where
\begin{align}\label{chap2-def-W-k-c}
W_{X_k} = \mathrm{span}\{\phi^{X_k}_{1},\cdots,\phi^{X_k}_{n_{X_k}}\},~X = F, E, V,
\end{align}
and the corresponding basis function vectors can be denoted as
\begin{align}\label{chap2-def-Phik}
\Phi_{X_k} = (\phi^{X_k}_{1},\cdots,\phi^{X_k}_{n_{X_k}})^T,~X = F, E, V,
\end{align}
respectively.

Define the Schur complement operator $\hat{S}:\hat{W} \rightarrow \hat{W}$ as
\begin{align}\label{chap2-def-operator-hat-S}
(\hat{S} \hat{u}, \hat{v}) = a(\hat{u}, \hat{v}),~~\forall \hat{u}, \hat{v} \in \hat{W}.
\end{align}
Noting that $a(\cdot,\cdot)$ is Hermitian positive definite in $\hat{W} \subset V_p(\mathcal{T}_h)$,
thus $\hat{S}$ is also Hermitian positive definite.

Using \eqref{chap2-def-operator-hat-S}, we can derive the corresponding Schur complement problem of \eqref{115-dis} as: find $\hat{w} \in \hat{W}$ such that
\begin{align}\label{chap2-def-schur-system}
(\hat{S} \hat{w}, \hat{v}) = \mathcal{L}(\hat{v}),~~\forall \hat{v} \in \hat{W}.
\end{align}

In the next subsection, we will propose a two-level adaptive BDDC algorithm in variational form for solving \eqref{chap2-def-schur-system}.

\subsection{Two-level methods}

In order to construct the two-level adaptive BDDC preconditioner, we need to describe another basis functions of $\hat{W}$,
and carry out an equivalent system of \eqref{chap2-def-schur-system}.

\subsubsection{Primal and dual spaces}

For any $\mathcal{X}_k$ and $\nu \in \mathcal{N}_{\mathcal{X}_k} (X = F, E, V)$, we introduce the basis functions
\begin{align}\label{chap2-def-phi-k-nu-l-n}
\phi^{X_k,\nu}_{l} =  \varphi^{(\nu)}_{k_l} + \check{\phi}^{X_k,\nu}_{l},~l=1,\cdots,n_{X_k},
\end{align}
where  $\check{\phi}^{X_k,\nu}_{l}$ is defined by \eqref{chap2-def-Hw-i}.

From \eqref{chap2-def-phi-k-l-n}, \eqref{chap2-def-phi-k-nu-l-n} and \eqref{chap2-def-trunc-func}, and using the support property of $\varphi_{k_l}^{(\nu)}$ and $\check{\phi}^{X_k,\nu}_{l}$, we have
\begin{align}\label{chap2-rel-phik-phikv}
\phi^{X_k}_l|_{\bar{D}_{\nu}} = \phi^{X_k,\nu}_l|_{\bar{D}_{\nu}},~l=1,\cdots,n_{X_k},~\nu \in \mathcal{N}_{\mathcal{X}_k}, ~X = F, E, V.
\end{align}

Using the above basis functions, the function spaces associated with $F_k$, $E_k$ and $V_k$ are defined as
\begin{align}\label{chap2-def-Wk-BarWk-nu}
W_{X_k}^{(\nu)}  &=\mathrm{span}\{\phi^{X_k,\nu}_{1}, \cdots, \phi^{X_k,\nu}_{n_{X_k}}\},~\nu \in \mathcal{N}_{\mathcal{X}_k},~X = F, E, V,
\end{align}
and set
\begin{align}\label{chap2-def-Phi-Psi-k-nu}
\Phi_{X_k}^{(\nu)} = (\phi^{X_k,\nu}_{1}, \cdots, \phi^{X_k,\nu}_{n_{X_k}})^T,~X = F, E, V.
\end{align}

For any subdomain $D_r$, let
\begin{eqnarray}\label{chap2-def-Wr-old}
W^{(r)} = (\oplus_{k \in \mathcal{M}_F^{(r)}} W_{F_k}^{(r)}) \oplus (\oplus_{k \in \mathcal{M}_E^{(r)}} W_{E_k}^{(r)}) \oplus (\oplus_{k \in \mathcal{M}_V^{(r)}} W_{V_k}^{(r)}),
\end{eqnarray}
and we denote the corresponding basis function vector as $\Phi^{(r)}$.

Then, we will define two transformation operators that change the basis functions $\{\phi_l^{X_k}\}_{l=1}^{n_{X_k}} (X = F, E)$ and $\{\phi_l^{X_k,\nu}\}_{l=1}^{n_{X_k}}(\nu \in \mathcal{N}_{\mathcal{X}_k}, X=F, E)$ into their corresponding dual-primal basis functions.

For $X=F,E$, set $n_{X_k}=n_{\Delta}^{X_k}+n_{\Pi}^{X_k}$, where the integers $n_{\Delta}^{X_k}, n_{\Pi}^{X_k} \ge 0$.
Let the nonsingular matrix $\vec{T}_{X_k} \in \mathbb{C}^{n_{X_k} \times n_{X_k}}$ look like
\begin{eqnarray}\label{chap2-ZJ-Oper-TFk-DelPi-Def}
\vec{T}_{X_k}=(\vec{T}_{\Delta}^{X_k},~ \vec{T}_{\Pi}^{X_k}),
\end{eqnarray}
where $\vec{T}_{\Delta}^{X_k} \in \mathbb{C}^{n_{X_k}\times n_{\Delta}^{X_k}}$ and $\vec{T}_{\Pi}^{X_k} \in \mathbb{C}^{n_{X_k}\times n_{\Pi}^{X_k}}$.

For any given $\nu \in \mathcal{N}_{\mathcal{X}_k} (X = F, E)$, by using $\vec{T}_{X_k}$ (or
~$\vec{T}_{\Delta}^{X_k}$ and~$\vec{T}_{\Pi}^{X_k}$), we can introduce the linear transformation operator
$T_{X_k}$ (or $T_{\Delta}^{X_k}$ and $T_{\Pi}^{X_k}$) which transform the basis vector $\Phi_{X_k}$ and $\Phi_{X_k}^{(\nu)}(\nu \in \mathcal{N}_{\mathcal{X}_k})$ associated with $W_{X_k}$ and $W_{X_k}^{(\nu)}$ into
\begin{eqnarray}\label{chap2-ZJ-Oper-TPhik-Def}
 \Phi_T^{X_k} = T_{X_k} \Phi_{X_k} :=
\left(\begin{array}{l}
\Phi^{X_k}_{\Delta}\\
\Phi^{X_k}_{\Pi}
\end{array}\right),~~
 \Phi_T^{X_k,\nu} = T_{X_k} \Phi_{X_k}^{(\nu)} :=
\left(\begin{array}{l}
\Phi^{X_k,\nu}_{\Delta}\\
\Phi^{X_k,\nu}_{\Pi}
\end{array}\right),
\end{eqnarray}
where
\begin{align}\label{chap2-ZJ-Def-Bases-WkDel-WkPi}
&\Phi^{X_k}_{\zeta} = (\phi^{X_k}_{\zeta,1},\cdots,\phi^{X_k}_{\zeta,n_{\zeta}^{X_k}})^T = T_{\zeta}^{X_k} \Phi_{X_k}:=
 (\vec{T}_{\zeta}^{X_k})^T \Phi_{X_k},~\zeta=\Delta,
\Pi,\\
\label{chap2-Zj-Def-Vec-TildPhiDel-1}
&\Phi^{X_k,\nu}_{\zeta} = (\phi^{X_k,\nu}_{\zeta,1},\cdots,\phi^{X_k,\nu}_{\zeta,n_{\zeta}^{X_k}})^T = T_{\zeta}^{X_k} \Phi_{X_k}^{(\nu)}:=
 (\vec{T}_{\zeta}^{X_k})^T \Phi_{X_k}^{(\nu)},~\zeta=\Delta,
\Pi.
\end{align}

Corresponding to \eqref{chap2-ZJ-Def-Bases-WkDel-WkPi},
\eqref{chap2-Zj-Def-Vec-TildPhiDel-1} and \eqref{chap2-rel-phik-phikv}, it follows that
\begin{align}\label{chap2-rel-phichik-phichikv}
\phi^{X_k}_{\zeta,l}|_{\bar{D}_{\nu}} = \phi^{X_k,\nu}_{\zeta,l}|_{\bar{D}_{\nu}},~l=1,\cdots,n_{\zeta}^{X_k},~\zeta = \Delta,\Pi,~\nu \in \mathcal{N}_{X_k},~X = F, E.
\end{align}

For an adaptive BDDC algorithm, the transformation operators are constructed by solving a generalized eigenvalue problem on each glob (or $\mathcal{X}_k, k=1,\cdots,N_X, X = F, E$).
The idea is originated from the upper bound estimate of BDDC preconditioned operator.

In the following, we present the generalized eigenvalue problems for the adaptive BDDC algorithm on each glob $\mathcal{X}_k (k=1,\cdots,N_X, X=F,E)$.
For that we introduce the scaling operators firstly.

For any given $\nu \in \mathcal{N}_{X_k}$, we present the scaling operator
$D_{X_{k}}^{(\nu)}: U \rightarrow  U$ $(U = W_{X_k}, W_{X_k}^{(\mu)}(\mu \in \mathcal{N}_{\mathcal{X}_k})$) or the scaling matrix
$\vec{D}_{X_{k}}^{(\nu)} \in \mathbb{C}^{n_{X_k}\times n_{X_k}}$, which satisfies that, for any given function $w = \vec{w}^T \Psi \in
U$, $\vec{w} \in \mathbb{C}^{n_{X_k}}$, $\Psi = \Phi_{X_k}, \Phi_{X_k}^{\mu}(\mu \in \mathcal{N}_{\mathcal{X}_k})$, we have
\begin{eqnarray}\label{chap2-def-D-operator-matrix}
D_{X_{k}}^{(\nu)} w = \vec{w}^T (\vec{D}_{X_{k}}^{(\nu)})^T \Psi,
\end{eqnarray}
where $\vec{D}_{X_{k}}^{(\nu)}$ is nonsingular, and
\begin{eqnarray}\label{chap2-def-D-trans}
\sum\limits_{\nu \in \mathcal{N}_{\mathcal{X}_k}} D_{X_{k}}^{(\nu)}  = I,
\end{eqnarray}
here $I$ is an identify operator in $U$.

Using the transformation operators and the scaling operators, a new set of basis functions of $W_{X_k}^{(\nu)} (\nu \in \mathcal{N}_{X_k})$ can be defined as
\begin{align}\label{chap2-ZJ-TBarphik-ij-Def}
T_{X_k} \Phi^{X_k,\nu,\mu}_D := \left(
\begin{array}{l}
\Phi^{X_k, \nu, \mu}_{D,\Delta}\\
\Phi^{X_k, \nu, \mu}_{D,\Pi}
\end{array}\right),
\end{align}
where
\begin{align}\label{chap2-ZJ-Dphik-ij-Def}
&\Phi^{X_k,\nu,\mu}_D = D_{X_{k}}^{(\mu)} \Phi_{X_k}^{(\nu)}:=(\vec{D}_{X_{k}}^{(\mu)})^T
\Phi_{X_k}^{(\nu)}, \\\label{chap2-ZJ-TDphik-ij-Def-00}
&\Phi^{X_k,\nu,\mu}_{D,\zeta}=T_{\zeta}^{X_k} \Phi^{X_k,\nu,\mu}_D := (\vec{T}_{\zeta}^{X_k})^T \Phi^{X_k,\nu,\mu}_D,~\zeta=\Delta,
\Pi.
\end{align}

Then, we introduce some auxiliary basis functions for each equivalence class and derive some related properties.

For any given $\mathcal{X}_k$ and $\nu \in \mathcal{N}_{\mathcal{X}_k}$, we introduce
\begin{align}\label{chap2-def-barphi-k-nu-l-n}
\bar{\phi}^{X_k,\nu}_l = \phi^{X_k,\nu}_{l} + \check{\bar{\phi}}^{X_k,\nu}_l,~l=1,\cdots,n_{X_k},
\end{align}
where $\check{\bar{\phi}}^{X_k,\nu}_l \in W^{(\nu)} \backslash W_{X_k}^{(\nu)}$ satisfies
\begin{align}\label{chap2-def-barphi-2}
a_{\nu}(\check{\bar{\phi}}^{X_k,\nu}_l, v) = -a_{\nu}(\phi^{X_k,\nu}_{l}, v), ~~\forall v \in W^{(\nu)} \backslash W_{X_k}^{(\nu)}.
\end{align}

Similar to Lemma 1 in \cite{PengWang18:683}, we can prove that $a_{\nu}(\cdot,\cdot)$ is positive definite in $W^{(\nu)} \backslash W_{X_k}^{(\nu)}$,
therefore, the basis functions $\{\bar{\phi}_l^{X_k,\nu}\}$ are existent and unique.

By using \eqref{chap2-def-Hw-i}, \eqref{chap2-def-phi-k-nu-l-n},  \eqref{chap2-def-barphi-k-nu-l-n} and  \eqref{chap2-def-barphi-2}, we can obtain
\begin{align}\label{chap2-def-barHw-i-Fim-22}
&\bar{\phi}^{X_k,\nu}_{l} - \phi^{X_k,\nu}_{l} \in V^{(\nu)} \backslash V_{X_k}^{(\nu)},~l=1,\cdots, n_{X_k},\\\label{chap2-def-barHw-i-Fim}
&a_{\nu}(\bar{\phi}^{X_k,\nu}_{l}, v) = 0,~\forall v \in V^{(\nu)} \backslash V_{X_k}^{(\nu)},~l=1,\cdots, n_{X_k}.
\end{align}

Using the basis functions $\{\bar{\phi}_l^{X_k,\nu}\}_{l=1}^{n_{X_k}} (\nu \in \mathcal{N}_{X_k})$, we can define the corresponding function spaces
\begin{align}\label{chap2-def-bar-W-k-nu}
 \bar{W}_{X_k}^{(\nu)}  =\mathrm{span}\{\bar{\phi}^{X_k,\nu}_{1}, \cdots, \bar{\phi}^{X_k,\nu}_{n_{X_k}}\},~\nu \in \mathcal{N}_{X_k},
\end{align}
and set
\begin{align}\label{chap2-def-barPhi-k-nu}
\bar{\Phi}_{X_k}^{(\nu)} = (\bar{\phi}^{X_k,\nu}_{1}, \cdots, \bar{\phi}^{X_k,\nu}_{n_{X_k}})^T, ~\nu \in \mathcal{N}_{X_k}.
\end{align}

Similar to the proof of Lemma 2 in \cite{PengWang18:683}, it can be proved that the following lemma holds.
\begin{lemma}\label{chap2-lemma-contral-iequ}
For any given subdomain $D_r$ and vectors $\{\vec{w}_{X_m}^{(r)} \in \mathbb{C}^{n_{X_m}},~m \in \mathcal{M}_X^{(r)}\} (X = F, E)$, we have
\begin{eqnarray}\label{chap2-4-1-2-tildew}
\sum\limits_{m \in \mathcal{M}_X^{(r)}} |\bar{w}_{X_m}^{(r)}|_{a_r}^2 \le |\mathcal{M}_X^{(r)}| |w^{(r)}|_{a_r}^2,~r=1,\cdots,N_d,
\end{eqnarray}
where
\begin{align*}
w^{(r)} &= \sum\limits_{m \in \mathcal{M}_F^{(r)}} w_{F_m}^{(r)} + \sum\limits_{m \in \mathcal{M}_E^{(r)}} w_{E_m}^{(r)} + \sum\limits_{m \in \mathcal{M}_V^{(r)}} w_{V_m}^{(r)},\\
\bar{w}_{X_m}^{(r)} &= (\vec{w}_{X_m}^{(r)})^T \bar{\Phi}_{X_m}^{(r)},~w_{X_m}^{(r)} = (\vec{w}_{X_m}^{(r)})^T \Phi_{X_m}^{(r)},~\forall m \in \mathcal{M}_X^{(r)},
\end{align*}
here $\mathcal{M}_F^{(r)}, \mathcal{M}_E^{(r)}, \mathcal{M}_V^{(r)}$ are defined in \eqref{chap2-def-M-i}, and if $X = F$ (or $X = E$), then $w_{\alpha_m}^{(r)} (\alpha \in \{F,E,V\}\backslash \{X\})$ is any given function in $W_{\alpha_m}^{(r)}$.
\end{lemma}

Then, we introduce matrices
\begin{eqnarray}\label{chap2-Sij-def}
\vec{S}^{(\nu)}_{X_k}  = (b_{l,m}^{(\nu)})_{n_{X_k}\times
n_{X_k}},~b_{l,m}^{(\nu)} = a_{\nu}(\phi^{X_k,\nu}_m,
\phi^{X_k,\nu}_l),~l,m=1,\cdots, n_{X_k},~\nu \in \mathcal{N}_{\mathcal{X}_k},\\\label{chap2-Sij-def-2}
\vec{\bar{S}}^{(\nu)}_{X_k}  = (c_{l,m}^{(\nu)})_{n_{X_k}\times
n_{X_k}},~c_{l,m}^{(\nu)} = a_{\nu}(\bar{\phi}^{X_k,\nu}_m,
\bar{\phi}^{X_k,\nu}_l),~l,m=1,\cdots, n_{X_k},~\nu \in \mathcal{N}_{\mathcal{X}_k}.
\end{eqnarray}
Since $a_{\nu}(\cdot,\cdot)$ is Hermitian semi-positive definite in $W^{(\nu)}$ and $\bar{\phi}^{X_k,\nu}_l \in W^{(\nu)}$,
we can see that the matrices $\vec{\bar{S}}_{X_{k}}^{(\nu)}(\nu \in \mathcal{N}_{\mathcal{X}_k})$ is Hermitian semi-positive definite.
Therefore, we can define the parallel sum matrix~(see~\cite{AndersonDuffin69:576} for more detials) $\prod_{\nu \in \mathcal{N}_{X_k}} \vec{\bar{S}}_{X_{k}}^{(\nu)}$.

From the properties of the parallel sum matrix, we know that $\prod_{\nu \in \mathcal{N}_{X_k}} \vec{\bar{S}}_{X_{k}}^{(\nu)}$ is Hermitian semi-positive definite and satisfies the following spectrum estimations
\begin{eqnarray}\label{chap2-parallel-property}
\prod_{\nu \in \mathcal{N}_{X_k}} \vec{\bar{S}}_{X_{k}}^{(\nu)} \le
\vec{\bar{S}}_{X_{k}}^{(\mu)},~\mu \in \mathcal{N}_{\mathcal{X}_k}.
\end{eqnarray}

Using~$\vec{S}_{X_{k}}^{(\nu)}$ and $\vec{\bar{S}}_{X_{k}}^{(\nu)}(\nu \in \mathcal{N}_{\mathcal{X}_k})$, we can introduce a generalized eigenvalue problem
\begin{eqnarray}\label{chap2-eig-pro-intro}
\vec{A}_{X_k}^D \vec{v} = \lambda \vec{B}_{X_k} \vec{v},~\vec{v} \in \mathbb{C}^{n_{X_k}},
\end{eqnarray}
where
\begin{eqnarray}\label{chap2-Def-A-B}
\vec{A}_{X_k}^D = \sum\limits_{\nu \in \mathcal{N}_{\mathcal{X}_k}} \sum\limits_{\mu \in \mathcal{N}_{\mathcal{X}_k} \backslash \{\nu\}} (\vec{D}_{X_k}^{(\mu)})^H \vec{S}_{X_k}^{(\nu)} \vec{D}_{X_k}^{(\mu)},~
\vec{B}_{X_k} = \prod_{\nu \in \mathcal{N}_{X_k}} \vec{\bar{S}}_{X_{k}}^{(\nu)},
\end{eqnarray}
here the eigenvalue $\lambda \in \mathbb{C}$, $\vec{D}_{X_k}^{(\nu)}(\nu \in \mathcal{N}_{\mathcal{X}_k})$ are the scaling matrices, $\diamond^H$ denotes the conjugate transpose of $\diamond$.

For a given real number $\Theta_X \ge 1$, we assume that the eigenvalues $\lambda_k (k=1,\cdots,n_{X_k})$ in \eqref{chap2-eig-pro-intro} satisfy
\begin{align}\label{chap2-lambda}
|\lambda_1| \le |\lambda_2| \le \cdots \le |\lambda_{n_{\Delta}^{X_k}}| \le \Theta_X \le |\lambda_{n_{\Delta}^{X_k}+1}| \le \cdots \le |\lambda_{n_{X_k}}|,
\end{align}
where $n_{\Delta}^{X_k}$ is a nonnegative integer.

We assume that the eigenvectors  $\vec{v}_l$ associated with $\lambda_l(l=1,\cdots,n_{X_k})$ have the orthogonal relation
\begin{align}\label{chap2-oth-1}
\vec{v}_l^H \vec{A}_{X_k}^D \vec{v}_m = \vec{v}_l^H \vec{B}_{X_k} \vec{v}_m = 0,~\mbox{if}~m \neq l.
\end{align}
and let the submatrices of $\vec{T}_{X_k}$ defined in \eqref{chap2-ZJ-Oper-TFk-DelPi-Def} be
\begin{align}\label{chap2-def-T-Delta1-Pi1}
\vec{T}_{\Delta}^{X_k}:= (\vec{v}_1,\cdots,\vec{v}_{n_{\Delta}^{X_k}}),~~\vec{T}_{\Pi}^{X_k}:= (\vec{v}_{n_{\Delta}^{X_k}+1},\cdots,\vec{v}_{n_{X_k}}).
\end{align}

By the definitions \eqref{chap2-ZJ-Oper-TFk-DelPi-Def} and \eqref{chap2-def-T-Delta1-Pi1} of the transform matrix $\vec{T}_{X_k}$,
the transformation operator $T_{X_k}$ can be obtained from \eqref{chap2-ZJ-Oper-TPhik-Def} and the following lemma holds.

\begin{lemma}
For the given real number $\Theta_X \ge 1 (X = F, E)$ in \eqref{chap2-lambda}, and assume that the transformation operator $T_{X_k}$ is given by \eqref{chap2-ZJ-Oper-TFk-DelPi-Def}, \eqref{chap2-ZJ-Oper-TPhik-Def} and \eqref{chap2-def-T-Delta1-Pi1}, then we can obtain the following estimate
 \begin{eqnarray}\label{chap2-ZJ-Tphik-TDphik-Prop}
\sum\limits_{s \in \mathcal{N}_{\mathcal{X}_k} \backslash \{r\}} (|w_{D,\Delta}^{X_k,r,s}|_{a_r}^2 +
|\tilde{w}_{D,\Delta}^{X_k,s,r}|_{a_s}^2) \le \Theta_X
|\bar{w}_{X_k,\Delta}^{(r)} +\bar{w}_{X_k,\Pi}^{(r)}|_{a_r}^2,~ \forall r \in \mathcal{N}_{\mathcal{X}_k},
\end{eqnarray}
where
\begin{eqnarray}\label{chap2-ZJ-Tphik-TDphik-w1}
w_{D,\Delta}^{X_k,r,s} =(\vec{w}_{\Delta})^T
\Phi^{X_k,r,s}_{D,\Delta}, \tilde{w}_{D,\Delta}^{X_k,s,r} =(\vec{w}_{\Delta})^T
\Phi^{X_k,s,r}_{D,\Delta}, \bar{w}_{X_k,\zeta}^{(r)}=
(\vec{w}_{\zeta})^T  \bar{\Phi}^{X_k,r}_{\zeta}, \zeta=\Delta, \Pi,
\end{eqnarray}
here $\Phi^{X_k,r,s}_{D,\Delta}$ and $\Phi^{X_k,s, r}_{D,\Delta}$ are defined in \eqref{chap2-ZJ-TDphik-ij-Def-00},
\begin{align}\label{chap2-ZJ-TDphik-ij-Def}
 \bar{\Phi}^{X_k,r}_{\zeta} = (\bar{\phi}^{X_k,r}_{\zeta,1},\cdots,\bar{\phi}^{X_k,r}_{\zeta,n_{\chi}^{X_k}})^T  = T_{\zeta}^{X_k} \bar{\Phi}_{X_k}^{(r)}:=
 (\vec{T}_{\zeta}^{X_k})^T  \bar{\Phi}_{X_k}^{(r)},~\zeta =\Delta,
\Pi,
\end{align}
and $\vec{w}_{X_k,\Delta} \in \mathbb{C}^{n_{\Delta}^{X_k}}$, $\vec{w}_{X_k,\Pi} \in \mathbb{C}^{n_{\Pi}^{X_k}}$
are any given vectors.
\end{lemma}

\begin{proof}
By using \eqref{chap2-oth-1}, we can obtain
\begin{eqnarray}\label{chap2-OrthProp-Tk-Delta}
&&(\vec{T}_{\Delta}^{X_k})^H \vec{C} \vec{T}_{\Delta}^{X_k}= diag
\{\vec{v}_1^H \vec{C} \vec{v}_1,\cdots, \vec{v}_{n_{\Delta}^{X_k}}^H \vec{C}
\vec{v}_{n_{\Delta}^{X_k}}\},~~\vec{C} = \vec{A}_{X_k}^D, \vec{B}_{X_k}, \\\label{chap2-OrthProp-Tk-Delta-Pi}
&&(\vec{T}_{\Pi}^{X_k})^H \vec{B}_{X_k}
\vec{T}_{\Delta}^{X_k}=0,~ (\vec{T}_{\Delta}^{X_k})^H \vec{B}_{X_k}
\vec{T}_{\Pi}^{X_k} = 0.
\end{eqnarray}
From \eqref{chap2-eig-pro-intro}, it is easily seen that
\begin{align}\label{chap2-remark-equ-general-eig-pro}
\vec{v}_l^H \vec{A}_{X_k}^D \vec{v}_l = |\lambda_l| \vec{v}_l^H \vec{B}_{X_k} \vec{v}_l,~l=1,\cdots,n_{\Delta}^{X_k}.
\end{align}

Combining \eqref{chap2-ZJ-TDphik-ij-Def} with \eqref{chap2-ZJ-TDphik-ij-Def-00},
we can separately rewrite the functions in \eqref{chap2-ZJ-Tphik-TDphik-w1} as
\begin{align}\label{chap2-Esp-wk-Del-Td-nu}
 w_{D,\Delta}^{X_k,r,s} &=(\vec{w}_{\Delta})^T
(\vec{T}_{\Delta}^{X_k})^T  (\vec{D}_{X_{k}}^{(s)})^T \Phi_{X_k}^{(r)},~
\tilde{w}_{D,\Delta}^{X_k,s,r} =(\vec{w}_{\Delta})^T
(\vec{T}_{\Delta}^{X_k})^T  (\vec{D}_{X_{k}}^{(r)})^T \Phi_{X_k}^{(s)},\\\label{chap2-Esp-Barwk-chi-Ti}
 \bar{w}_{X_k,\zeta}^{(r)} &= (\vec{w}_{\zeta})^T (\vec{T}_{\zeta}^{X_k})^T  \bar{\Phi}_{X_k}^{(\nu)},~\zeta=\Delta, \Pi.
\end{align}

According to \eqref{def-normAi}, \eqref{chap2-Esp-wk-Del-Td-nu} and \eqref{chap2-Sij-def}, we have
\begin{align}\nonumber \label{chap2-Ai-Dkj-wkDeltai}
|w_{D,\Delta}^{X_k,r,s}|_{a_r}^2 
&= \vec{w}_{\Delta}^H(\vec{T}_{\Delta}^{X_k})^H (\vec{D}_{X_k}^{(s)})^H a_r(\Phi_{X_k}^{(r)},(\Phi_{X_k}^{(r)})^T)
 \vec{D}_{X_k}^{(s)} \vec{T}_{\Delta}^{X_k} \vec{w}_{\Delta}\\
&= \vec{w}_{\Delta}^H(\vec{T}_{\Delta}^{X_k})^H (\vec{D}_{X_k}^{(s)})^H\vec{S}_{X_k}^{(r)}
 \vec{D}_{X_k}^{(s)} \vec{T}_{\Delta}^{X_k} \vec{w}_{\Delta}.
\end{align}

Similarly, we can obtain
 \begin{eqnarray}\label{chap2-Aj-Dki-wkDeltaj}
|\tilde{w}_{D,\Delta}^{X_k,s,r}|_{a_s}^2
= \vec{w}_{\Delta}^H (\vec{T}_{\Delta}^{X_k})^H (\vec{D}_{X_k}^{(r)})^H \vec{S}_{X_k}^{(s)}
 \vec{D}_{X_k}^{(r)} \vec{T}_{\Delta}^{X_k} \vec{w}_{\Delta}.
\end{eqnarray}

Using \eqref{chap2-Ai-Dkj-wkDeltai}, \eqref{chap2-Aj-Dki-wkDeltaj}, \eqref{chap2-Def-A-B}, \eqref{chap2-OrthProp-Tk-Delta} and noting that $\vec{S}_{X_k}^{(\mu)} (\mu \in \mathcal{N}_{\mathcal{X}_k})$ are semi-positive definite, we have
\begin{align*} 
\sum\limits_{s \in \mathcal{N}_{\mathcal{X}_k} \backslash \{r\}} (|w_{D,\Delta}^{X_k,r,s}|_{a_r}^2 +
|\tilde{w}_{D,\Delta}^{X_k,s,r}|_{a_s}^2)
&\le \vec{w}_{\Delta}^H diag\{\vec{v}_1^H \vec{A}_{X_k}^D \vec{v}_1,\cdots, \vec{v}_{n_{\Delta}^{X_k}}^H \vec{A}_{X_k}^D \vec{v}_{n_{\Delta}^{X_k}}\}\vec{w}_{\Delta}.
\end{align*}

From this, using \eqref{chap2-remark-equ-general-eig-pro}, \eqref{chap2-lambda}, \eqref{chap2-OrthProp-Tk-Delta}, \eqref{chap2-OrthProp-Tk-Delta-Pi}, \eqref{chap2-parallel-property},  \eqref{chap2-Sij-def-2}, \eqref{chap2-Esp-Barwk-chi-Ti} and \eqref{def-normAi}, and noting $\vec{B}_{X_k}$ is Hermitian semi-positive definite, it follows that
\begin{align*}
\sum\limits_{s \in \mathcal{N}_{\mathcal{X}_k} \backslash \{r\}} (|w_{D,\Delta}^{X_k,r,s}|_{a_r}^2 +
|\tilde{w}_{D,\Delta}^{X_k,s,r}|_{a_s}^2)
&\le \vec{w}_{\Delta}^H diag\{|\lambda_1| \vec{v}_1^H \vec{B}_{X_k} \vec{v}_1,\cdots, |\lambda_{n_{\Delta}^{X_k}}| \vec{v}_{n_{\Delta}^{X_k}}^H \vec{B}_{X_k} \vec{v}_{n_{\Delta}^{X_k}}\}\vec{w}_{\Delta}\\\nonumber
&\le \Theta_X \vec{w}_{\Delta}^H diag\{\vec{v}_1^H \vec{B}_{X_k} \vec{v}_1,\cdots, \vec{v}_{n_{\Delta}^{X_k}}^H \vec{B}_{X_k} \vec{v}_{n_{\Delta}^{X_k}}\}\vec{w}_{\Delta}\\
&= \Theta_X \vec{w}_{\Delta}^H (\vec{T}_{\Delta}^{X_k})^H \vec{B}_{X_k} \vec{T}_{\Delta}^{X_k} \vec{w}_{\Delta}\\
&\le \Theta_X (\vec{T}_{\Delta}^{X_{k}}\vec{w}_{\Delta} + \vec{T}_{\Pi}^{X_{k}}\vec{w}_{\Pi})^H \vec{B}_{X_k} (\vec{T}_{\Delta}^{X_{k}}\vec{w}_{\Delta} + \vec{T}_{\Pi}^{X_{k}}\vec{w}_{\Pi})\\\nonumber
&\le \Theta_X (\vec{T}_{\Delta}^{X_{k}}\vec{w}_{\Delta} + \vec{T}_{\Pi}^{X_{k}}\vec{w}_{\Pi})^H
\vec{\bar{S}}_{X_{k}}^{(r)} (\vec{T}_{\Delta}^{X_{k}}\vec{w}_{\Delta} + \vec{T}_{\Pi}^{X_{k}}\vec{w}_{\Pi})\\\nonumber
&= \Theta_X a_r(\bar{w}_{X_k,\Delta}^{(r)} + \bar{w}_{X_k,\Pi}^{(r)}, \bar{w}_{X_k,\Delta}^{(r)} + \bar{w}_{X_k,\Pi}^{(r)})\\
&= \Theta_X |\bar{w}_{X_k,\Delta}^{(r)} + \bar{w}_{X_k,\Pi}^{(r)}|_{a_r}^2,
\end{align*}
then \eqref{chap2-ZJ-Tphik-TDphik-Prop} holds.
\end{proof}

\begin{remark}
Although the condition number of the preconditioned operator can be controlled by the user-defined tolerances,
the cost for forming the two classes of generalized eigenvalue problems is quite expensive.
Therefore, similar to \cite{KlawonnRadtke16:75}, we use economic-version to enhance the efficiency of the proposed method.
\end{remark}

From this, we can decompose the function spaces $W_{X_k}$ and $W_{X_k}^{(\nu)} (X = F, E)$ as
\begin{eqnarray}\label{chap2-ZJ-Def-Wk}
W_{X_k} = W_{X_k,\Delta}\oplus W_{X_k,\Pi},~
W_{X_k}^{(\nu)} = W_{X_k,\Delta}^{(\nu)} \oplus W_{X_k,\Pi}^{(\nu)},
\end{eqnarray}
where
\begin{align*}
&W_{X_k,\zeta} = \mathrm{span}\{\phi^{X_k}_{\zeta,1},\cdots,\phi^{X_k}_{\zeta,n_{\zeta}^{X_k}}\},~W_{X_k,\zeta}^{(\nu)} = \mathrm{span}\{\phi^{X_k,\nu}_{\zeta,1},\cdots,\phi^{X_k,\nu}_{\zeta,n_{\zeta}^{X_k}}\},~\zeta = \Delta,\Pi, ~\nu \in \mathcal{N}_{X_k}.
\end{align*}

For any given subdomain $D_r$, we define the function spaces
\begin{align}\label{chap2-def-Wr-DP}
&W_{\Delta}^{(r)} =(\oplus_{k\in \mathcal{M}_F^{(r)}} W_{F_k,\Delta}^{(r)}) \oplus (\oplus_{k\in \mathcal{M}_E^{(r)}} W_{E_k,\Delta}^{(r)}), \\\label{chap2-def-Wr-DP-Pi}
&W_{\Pi}^{(r)} = (\oplus_{k \in \mathcal{M}_F^{(r)}} W_{F_k,\Pi}^{(r)}) \oplus (\oplus_{k \in \mathcal{M}_E^{(r)}} W_{E_k,\Pi}^{(r)}) \oplus (\oplus_{k \in \mathcal{M}_V^{(r)}} W_{V_k}^{(r)}),
\end{align}
and let
\begin{eqnarray}\label{chap2-def-Wr}
\tilde{W}^{(r)} = W_{\Delta}^{(r)} \oplus W_{\Pi}^{(r)}.
\end{eqnarray}

Using the function spaces $W_{X_k,\zeta} (k=1,\cdots,N_X, \zeta = \Delta, \Pi, X = F, E)$ and $W_{V_k} (k=1,\cdots,N_V)$, we can define
\begin{eqnarray}\label{chap2-Zj-Def-WDelPi}
W_{\Delta}=(\oplus_{k=1}^{N_F} W_{F_k,\Delta}) \oplus (\oplus_{k=1}^{N_E} W_{E_k,\Delta}) ,~W_{\Pi} = (\oplus_{k=1}^{N_F} W_{F_k,\Pi}) \oplus (\oplus_{k=1}^{N_E} W_{E_k,\Pi}) \oplus (\oplus_{k=1}^{N_V} W_{V_k}),
\end{eqnarray}
where the function space $W_{\Pi}$ is the so-called primal space, and the corresponding basis functions $\{\phi_{\Pi,l}^{X_k}\} (X = F,E)$ and $\{\phi_{l}^{V_k}\}$ are the primal basis functions.

Similarly, by using the function spaces $W_{\Delta}^{(r)} (r=1,\cdots,N_d)$, we can define the so-called dual space
\begin{eqnarray}\label{chap2-ZJ-TildeW-Delta-Decomp}
\tilde{W}_{\Delta} = \oplus_{r=1}^{N_d} W_{\Delta}^{(r)}.
\end{eqnarray}

From \eqref{chap2-def-W}, \eqref{chap2-ZJ-Def-Wk} and \eqref{chap2-Zj-Def-WDelPi}, we can decompose the function space $\hat{W}$ which the Schur complement system depends on
into
\begin{eqnarray}\label{chap2-Zj-Decomp-HatW}
\hat{W} = W_{\Delta} \oplus W_{\Pi}.
\end{eqnarray}

By using the function spaces $W_{\Pi}$ and $\tilde{W}_{\Delta}$ defined by \eqref{chap2-Zj-Def-WDelPi} and \eqref{chap2-ZJ-TildeW-Delta-Decomp}, respectively, we can define a partial coupling function space which the adaptive BDDC preconditioner is based on
\begin{eqnarray}\label{chap2-ZJ-TildeW-Decomp}
\tilde{W} = \tilde{W}_{\Delta} \oplus W_{\Pi}.
\end{eqnarray}
We can see that the functions belonging to $\tilde{W}$ are continuous at the primal level and discontinuous elsewhere on the $\mathcal{F}_I$.

For convenience, we call the process of generating the required function spaces by using
the function space $V_p(\mathcal{T}_h)$ as~{\bf Setup algorithm}.

From now on, the function space $\hat{W}$ defined by \eqref{chap2-Zj-Decomp-HatW} will be adopted,
and we denote the corresponding Schur complement variational problem as (we still use the same notation as \eqref{chap2-def-schur-system} when no confusion can arise): find  $\hat{w} \in \hat{W}$ such that
\begin{align}\label{chap2-ZJ-Schur-System}
(\hat{S} \hat{w}, \hat{v}) = \mathcal{L}(\hat{v}),~~\forall \hat{v} \in \hat{W}.
\end{align}

In the following subsections we will design and analysis the adaptive BDDC preconditioner for solving the Schur complement system \eqref{chap2-ZJ-Schur-System}.

\subsubsection{BDDC preconditioner}

In order to derive an adaptive BDDC algorithm in variational form for solving the Schur complement system \eqref{chap2-ZJ-Schur-System},
some operators are introduced firstly.


%
We recall that $\phi^{(r)}\in \tilde{W}^{(r)}$ is the truncated function of the basis $\phi \in \hat{W}$ in $D_r$.
For any given $\phi^{(r)} \in \tilde{W}^{(r)} (r=1,\cdots,N_d)$, we introduce a basis transformation operator $T_r: \tilde{W}^{(r)} \rightarrow \hat{W}$, which is a linear operator and satisfies
\begin{align}\label{chap2-Basis-Tran-2}
T_r \phi^{(r)} = \phi,~~\forall~ r = 1,\cdots,N_d.
\end{align}
Conversely, for any given basis function $\phi \in \hat{W}$, a basis transformation operator $T^{r}: \hat{W} \rightarrow \tilde{W}^{(r)}$ is defined as
\begin{align}\label{chap2-Basis-Tran-1}
T^{r} \phi = \phi^{(r)},~~\forall~ r = 1,\cdots,N_d.
\end{align}

According to the decomposition \eqref{chap2-ZJ-TildeW-Decomp},
\eqref{chap2-Zj-Def-WDelPi} and \eqref{chap2-ZJ-TildeW-Delta-Decomp} of $\tilde{W}$, it follows that
\begin{eqnarray}\label{chap2-ZJ-TildeW-FuncDecomp}
\tilde{\zeta} = \sum\limits_{r=1}^{N_d} (\sum\limits_{k\in
\mathcal{M}_F^{(r)}} \tilde{\zeta}_{F_k,\Delta}^{(r)} + \sum\limits_{k\in
\mathcal{M}_E^{(r)}} \tilde{\zeta}_{E_k,\Delta}^{(r)}) + \sum\limits_{k=1}^{N_F}
\tilde{\zeta}_{F_k,\Pi} + \sum\limits_{k=1}^{N_E}
\tilde{\zeta}_{E_k,\Pi} + \sum\limits_{k=1}^{N_V}
\tilde{\zeta}_{V_k},
\end{eqnarray}
where $\tilde{\zeta}_{X_k,\Delta}^{(r)}\in W_{X_k,\Delta}^{(r)}$,
$\tilde{\zeta}_{X_k,\Pi}\in W_{X_k,\Pi} (X = F,E)$, $\tilde{\zeta}_{V_k} \in W_{V_k}$ for any given function $\tilde{\zeta}\in \tilde{W}$.

Using the above decomposition and $a_r(\cdot,\cdot) (r=1,\cdots,N_d)$ defined in \eqref{chap4-def-AUV}, we can obtain that for any given $\tilde{u},
\tilde{v} \in \tilde{W}$, we can define a semilinear form $\tilde{a}(\cdot,\cdot)$ and the corresponding partially assembled Schur complement operator $\tilde{S}: \tilde{W} \rightarrow \tilde{W}$ satisfing
\begin{eqnarray}\label{chap2-def-operator-tilde-A}
(\tilde{S} \tilde{u}, \tilde{v}) :=\tilde{a}(\tilde{u},
\tilde{v}) := \sum\limits_{r=1}^{N_d} a_r(\tilde{u}^{(r)},
\tilde{v}^{(r)}),~~\forall \tilde{u}, \tilde{v} \in
\tilde{W},
\end{eqnarray}
where
\begin{eqnarray}\label{chap2-def-operator-tilde-S}
\tilde{\zeta}^{(r)} = \sum\limits_{k \in \mathcal{M}_F^{(r)}}
(\tilde{\zeta}_{F_k,\Delta}^{(r)} + T^{r}
\tilde{\zeta}_{F_k,\Pi}) + \sum\limits_{k \in \mathcal{M}_E^{(r)}}
(\tilde{\zeta}_{E_k,\Delta}^{(r)} + T^{r}
\tilde{\zeta}_{E_k,\Pi}) + \sum\limits_{k\in \mathcal{M}_V^{(r)}} T^{r}
\tilde{\zeta}_{V_k},~~\tilde{\zeta} = \tilde{u}, \tilde{v}.
\end{eqnarray}

From the property that the semilinear form $\tilde{a}(\cdot,\cdot)$ is Hermitian positive definite in $\tilde{W}$, we can see the operator $\tilde{S}$ is also Hermitian positive definite.

Let $I_{\Gamma}: \tilde{W} \rightarrow \hat{W}$ be a linear operator which satisfies
\begin{align}\label{chap2-IGamma-1}
& I_{\Gamma} \phi_{\Delta,l}^{X_k,\nu} = \phi_{\Delta,l}^{X_k},~l=1,\cdots,n_{\Delta}^{X_k}, \nu \in \mathcal{N}_{\mathcal{X}_k}, k=1,\cdots,N_X, X=F,E, \\\label{chap2-IGamma-2}
& I_{\Gamma} \phi_{\Pi,l}^{X_k} = \phi_{\Pi,l}^{X_k},~l=1,\cdots,n_{\Pi}^{X_k}, k=1,\cdots,N_X, X=F,E, \\\label{chap2-IGamma-3}
& I_{\Gamma} \phi_{l}^{V_k} = \phi_{l}^{V_k},~l=1,\cdots,n_{V_k}, k=1,\cdots,N_V.
\end{align}


Using the scaling operators defined in \eqref{chap2-def-D-operator-matrix}, a linear operator $D_{\Delta}^{(r)}: W_{\Delta}^{(r)} \rightarrow W_{\Delta}^{(r)}$ for any given subdomain $D_r$ can be defined as
\begin{eqnarray}\label{chap2-ZJ-Def-Oper-HatI-DelDi}
D_{\Delta}^{(r)}
=\sum\limits_{k \in \mathcal{M}_F^{(r)}}
D_{F_k}^{(r)} R_{F_k,\Delta}^{(r)} +
\sum\limits_{k \in \mathcal{M}_E^{(r)}}
D_{E_k}^{(r)} R_{E_k,\Delta}^{(r)},
\end{eqnarray}
where $R_{X_k,\Delta}^{(r)} (X = F, E)$ are the restriction operators from $W_{\Delta}^{(r)}$ to $W_{X_k,\Delta}^{(r)}$.

According to \eqref{chap2-ZJ-Def-Oper-HatI-DelDi}, it is easy to verify that
\begin{eqnarray}\label{chap2-def-R-property-1}
D_{\Delta}^{(r)} w =  D_{X_k}^{(r)} w,~\forall w  \in W_{X_k,\Delta}^{(r)}
\end{eqnarray}
for any given $k \in \mathcal{M}_X^{(r)} (X = F, E)$.

By using the operators $D_{\Delta}^{(r)} (r=1,\cdots,N_d)$ defined in \eqref{chap2-ZJ-Def-Oper-HatI-DelDi}, we can introduce the scaling operator
$\tilde{D}$ from $\tilde{W}$ to $\tilde{W}$, which satisfies
\begin{align}\label{def-tildeD}
\tilde{D} = \sum\limits_{r=1}^{N_d} D_{\Delta}^{(r)}\tilde{R}_{\Delta}^{(r)} + \tilde{R}_{\Pi},
\end{align}
where $\tilde{R}_{\Delta}^{(r)}$ is the restriction operator from $\tilde{W}$ to $W_{\Delta}^{(r)}$ and $\tilde{R}_{\Pi}$ is the restriction operator from $\tilde{W}$ to $W_{\Pi}$.

With the above preparations, using the sesquilinear form $\tilde{a}(\cdot,\cdot)$,
an adaptive BDDC operator $M_{BDDC}^{-1}: \hat{W} \rightarrow \hat{W}$ for solving the Schur complement system \eqref{chap2-ZJ-Schur-System}  can then be given as the following algorithm.
\begin{algorithm}\label{chap2-ZJ-algorithm-variation-BDDC-domain-simple}~

For any given function $g\in \hat{W}$, $u_g = M_{BDDC}^{-1} {g} \in \hat{W}$ can be obtained by the following two steps:
\begin{description}
\item[Step 1.] Find $w \in \tilde{W}$, such that
\begin{align*}
\tilde{a}(w, v) = ((I_{\Gamma} \tilde{D})^H g, v),~~\forall v \in \tilde{W},
\end{align*}
where $(\diamond)^H$ denotes the complex conjugate transpose operator of $\diamond$.

\item[Step 2.] Let
\begin{align*}
u_g = I_{\Gamma} \tilde{D} w.
\end{align*}

\end{description}
\end{algorithm}

From this, combining with the definition \eqref{chap2-def-operator-tilde-A} of $\tilde{S}$ and noting that $\tilde{S}$ is Hermitian positive definite, the preconditioner operator $M_{BDDC}^{-1}$ can be showed as
\begin{align}\label{chap2-ZJ-Expression-MBDDC-Inv}
M_{BDDC}^{-1} = (I_{\Gamma} \tilde{D}) \tilde{S}^{-1} (I_{\Gamma} \tilde{D})^H.
\end{align}

In order to facilitate parallel programming, we want to give an equivalent description of Algorithm \ref{chap2-ZJ-algorithm-variation-BDDC-domain-simple}.
To this end, we need to introduce some other operators firstly.

Let $\tilde{I}_{\Delta}^{(r)}$ ($r=1,\cdots,N_d$) be the prolongation operators from $W_{\Delta}^{(r)}$ to $\tilde{W}$,
we introduce a linear operator from $\tilde{W}$ to $\tilde{W}_{\Delta}$ as
\begin{eqnarray*}
\tilde{S}_{\Delta}^{-1}= \sum\limits_{r=1}^{N_d}
(\tilde{S}^{(r)}_{\Delta\Delta})^{-1}(\tilde{I}_{\Delta}^{(r)})^H = \sum\limits_{r=1}^{N_d}
\tilde{I}_{\Delta}^{(r)}
(\tilde{S}^{(r)}_{\Delta\Delta})^{-1}(\tilde{I}_{\Delta}^{(r)})^H,
\end{eqnarray*}
where
\begin{align*}
\tilde{S}^{(r)}_{\Delta\Delta} =
(\tilde{I}_{\Delta}^{(r)})^H \tilde{S}
\tilde{I}_{\Delta}^{(r)},~r=1,\cdots,N_d.
\end{align*}

Using $\tilde{S}_{\Delta}^{-1}$, a linear operator from $W_{\Pi}$ to $\tilde{W}$ can be defined as
\begin{align}\label{chap2-ZJ-Repres-OPi-Eqiv}
O_{\tilde{\Pi}}= \tilde{I}_{\Pi} - \tilde{S}_{\Delta}^{-1} \tilde{S}
\tilde{I}_{\Pi} = (I - \tilde{S}_{\Delta}^{-1} \tilde{S}) \tilde{I}_{\Pi},
\end{align}
where $I: \tilde{W} \rightarrow \tilde{W}$ is an identity operator, and $\tilde{I}_{\Pi}: W_{\Pi} \rightarrow \tilde{W}$ is a prolongation operator.

Therefore, by using the expression \eqref{chap2-ZJ-Expression-MBDDC-Inv} of the preconditioner operator $M_{BDDC}^{-1}$,
after detailed deduction, we can arrived at the equivalent description of Algorithm \ref{chap2-ZJ-algorithm-variation-BDDC-domain-simple} as follows:
\begin{algorithm}\label{chap2-ZJ-algorithm-variation-BDDC-domain}~

For any given function $g\in \hat{W}$, $u_g = M_{BDDC}^{-1} {g}
\in \hat{W}$ can be obtained by the following steps:
\begin{description}
\item[Step 1.] Find $u^{\Delta,r}_a \in W_{\Delta}^{(r)}(r=1,\cdots,N_d)$ in parallel such that
\begin{align*}
a_r(u^{\Delta,r}_a, v) = ((T_r D_{\Delta}^{(r)})^H {g}, v),~~\forall v \in W_{\Delta}^{(r)},
\end{align*}
and compute
\begin{align*}
u_{\Delta,a} = \sum\limits_{r=1}^{N_d} T_r D_{\Delta}^{(r)} u^{\Delta,r}_a \in \hat{W},
\end{align*}
where the operators  $T_r$ and $D_{\Delta}^{(r)}$  are defined in \eqref{chap2-Basis-Tran-2} and \eqref{chap2-ZJ-Def-Oper-HatI-DelDi}.

\item[Step 2.]  Find $u_{\Pi} \in W_{\Pi}$ such that
\begin{align*}
\tilde{a}(O_{\tilde{\Pi}} u_{\Pi}, O_{\tilde{\Pi}}v) = (g, v) - \tilde{a}(\sum\limits_{r=1}^{N_d} u^{\Delta,r}_a,v),~~\forall v\in
W_{\Pi},
\end{align*}
where the operator $O_{\tilde{\Pi}}$ is defined in \eqref{chap2-ZJ-Repres-OPi-Eqiv}.

\item[Step 3.] Compute $u^{\Delta,r}_{b} \in W_{\Delta}^{(r)}(r=1,\cdots,N_d)$ in parallel by
\begin{align*}
a_r(u^{\Delta,r}_{b}, v) = - a_r(u_{\Pi}, v),~~\forall v \in W_{\Delta}^{(r)},
\end{align*}
and set
\begin{align*}
u_{\Delta,b} = \sum\limits_{r=1}^{N_d} T_r D_{\Delta}^{(r)} u^{\Delta,r}_b \in \hat{W}.
\end{align*}

\item[Step 4.] Let
\begin{align*}
u_g = u_{\Delta,a} + u_{\Pi} + u_{\Delta,b}.
\end{align*}

\end{description}

\end{algorithm}

Since Algorithm \ref{chap2-ZJ-algorithm-variation-BDDC-domain} is a two-level algorithm, we will call Algorithm \ref{chap2-ZJ-algorithm-variation-BDDC-domain} or its equivalent algorithm (Algorithm \ref{chap2-ZJ-algorithm-variation-BDDC-domain-simple}) as {\bf two-level adaptive BDDC algorithm}.

Furthermore, from Algorithm \ref{chap2-ZJ-algorithm-variation-BDDC-domain-simple} or Algorithm \ref{chap2-ZJ-algorithm-variation-BDDC-domain}, an algorithm for solving the original variational problem \eqref{115-dis} can be obtained.
\begin{algorithm}\label{chap2-ZJ-algorithm-variation-BDDC-domain-II}~

\begin{description}
\item[Step 1.]  By using the Krylov subspace iteration method based on preconditioner $M_{BDDC}^{-1}$, we can find $u_{\Gamma} \in \hat{W}$ such that
\begin{align*}
a(u_{\Gamma}, v) = \mathcal{L}(v),~\forall v \in \hat{W}.
\end{align*}

\item[Step 2.]  Compute $u^{(r)}_I \in V_I^{(r)}(r=1,\cdots,N_d)$ in parallel by
\begin{align*}
a_r(u^{(r)}_I, v) = \mathcal{L}(v) - a_r(u_{\Gamma}, v),~~\forall v \in V_I^{(r)}.
\end{align*}

\item[Step 3.] Set
\begin{align*}
u = \sum\limits_{r=1}^{N_d} u^{(r)}_I + u_{\Gamma}.
\end{align*}
\end{description}
\end{algorithm}
For the sake of description convenience, we call the above algorithm as {\bf two-level adaptive BDDC solver algorithm}.

\subsection{Multilevel extensions}

It is well known that the number of primal dofs in a nonoverlapping domain decomposition method will increase significantly as the number of subdomain increases.
The direct method is very expensive to solve the corresponding coarse problem.
In particular, for the Helmholtz problem with high-wave number, this phenomenon becomes more obvious with the increase of wave number.
The form of the coarse problem naturally leads to a multilevel extension of the BDDC algorithm \cite{Dohrmann03:246},
this can be used to overcome this difficulty efficiently.

In the following, based on Algorithm \ref{chap2-ZJ-algorithm-variation-BDDC-domain-II}, a rough description of the multi-level adaptive BDDC algorithm for solving the original variational problem \eqref{115-1} is given.

Firstly, we generate the mesh information for each level. Let $L$ be the total number of levels, and set the finest level by $s=0$.
We denote by $\mathcal{T}_h^s (s=0,\cdots,L-1)$ the mesh generation in the $s$-th level, respectively.
Let $\mathcal{T}_d^s (s=0,\cdots,L-2)$ be the subdomain generation in the $s$-th level, and satisfy $\mathcal{T}_h^0 = \mathcal{T}_h$, $\mathcal{T}_d^0 = \mathcal{T}_d$, $\mathcal{T}_h^{s+1} = \mathcal{T}_d^{s} (s \ge 0)$, $\mathcal{T}_d^{s}$ and $\mathcal{T}_h^{s} (s \ge 1)$ are nested.
Secondly, we generate the function spaces required by each level. Take the function space $V(\mathcal{T}_h^1)$ as the coarse space $W_{\Pi}$ (defined in \eqref{chap2-Zj-Def-WDelPi}) of the $0$-th level,
and regard $V(\mathcal{T}_h^1)$ as $V(\mathcal{T}_h)$ in the setup algorithm, we can obtain the corresponding coarse space $W_{\Pi}^1$ and other function spaces required in the $1$-th level;
In general, take the function spaces $V(\mathcal{T}_h^{s+1}) (s \ge 0)$ as the coarse space $W_{\Pi}^s$ of the $s$-th level,
and regard $V(\mathcal{T}_h^{s+1})$ as $V(\mathcal{T}_h)$ in the setup algorithm, we can reach the corresponding coarse space $W_{\Pi}^{s+1}$ and the other function spaces in the $s+1$-th level;
This process is executed sequentially until $s <= L-2$, we obtain the required function spaces on each level.
Based on the above preparations, and regard Algorithm \ref{chap2-ZJ-algorithm-variation-BDDC-domain-II} as an iterative algorithm from $s$-th level to $s$+$1$-th ($s = 0$) level,  a multi-level adaptive BDDC algorithm with the total number $L$ is obtained by calling the algorithm recursively until $s < L-1$.

Such an approach requires less memory than a two-level method with a direct coarse solver, and it can lead to highly scalable algorithms. Theoretically, the condition number of multilevel BDDC method depends multiplicatively on the condition number of each level problems \cite{MandelSousedik08:55}.

In the next section, we will derive the condition number estimation of the two-level adaptive BDDC preconditioned operator.

\section{Theoretical estimates}\setcounter{equation}{0}

In this section, we will provide the condition number estimate for the BDDC preconditioned operator with adaptive coarse space.

Let $\tilde{R}_{\Gamma}: \hat{W} \rightarrow \tilde{W}$ be the natural injection from $\hat{W}$ to $\tilde{W}$. It follows from \eqref{chap2-def-operator-hat-S} and \eqref{chap2-def-operator-tilde-A} that
\begin{align}\label{chap2-ZJn-Relat-hatS-tildeS}
 \hat{S} = (\tilde{R}_{\Gamma})^H \tilde{S}
 \tilde{R}_{\Gamma}.
\end{align}

Using \eqref{chap2-ZJn-Relat-hatS-tildeS} and \eqref{chap2-ZJ-Expression-MBDDC-Inv}, we can obtain the preconditioned operator associated with the Schur complement system \eqref{chap2-ZJ-Schur-System} as
\begin{align}\label{chap2-ZJ-oper-hatG}
\hat{G} := M_{BDDC}^{-1} \hat{S} = (I_{\Gamma} \tilde{D}) \tilde{S}^{-1} (I_{\Gamma} \tilde{D})^H
(\tilde{R}_{\Gamma})^H \tilde{S}
\tilde{R}_{\Gamma}.
\end{align}

In the following, we will derive the estimation of the minimum eigenvalue of the preconditioned operator $\hat{G}$.
For this purpose, the following lemma is given firstly.
\begin{lemma}\label{chap2-lemma-5-1}
Let $\tilde{R}_{\Gamma}: \hat{W} \rightarrow \tilde{W}$ be the natural injection from $\hat{W}$ to $\tilde{W}$, the linear operator $I_{\Gamma}: \tilde{W} \rightarrow \hat{W}$ is given by \eqref{chap2-IGamma-1}--\eqref{chap2-IGamma-3}, $\tilde{D}:\tilde{W} \rightarrow \tilde{W}$ is defined as \eqref{def-tildeD}, then we have
\begin{align}\label{chap2-lemma-R-hatW-tildeW-trans}
I_{\Gamma} \tilde{D} \tilde{R}_{\Gamma}  = I,
\end{align}
where $I: \hat{W} \rightarrow \hat{W}$ is an identity operator.
\end{lemma}

\begin{proof}
For any given function $\hat{u} \in \hat{W}$,  we can decompose $\hat{u}$ by using \eqref{chap2-Zj-Decomp-HatW} as
\begin{align}\label{chap2-decomposition-hatw}
\hat{u} = \sum\limits_{k=1}^{N_F} (\hat{u}_{F_k,\Delta} + \hat{u}_{F_k,\Pi}) + \sum\limits_{k=1}^{N_E} (\hat{u}_{E_k,\Delta} + \hat{u}_{E_k,\Pi}) + \sum\limits_{k=1}^{N_V} \hat{u}_{V_k},
\end{align}
where $\hat{u}_{F_k,\Delta} \in W_{F_k,\Delta}$, $\hat{u}_{F_k,\Pi}\in W_{F_k,\Pi}$, $\hat{u}_{E_k,\Delta} \in W_{E_k,\Delta}$, $\hat{u}_{E_k,\Pi}\in W_{E_k,\Pi}$ and ~$\hat{u}_{V_k} \in W_{V_k}$.

From this, and combining with the definitions of $\tilde{R}_{\Gamma}$ and $T^r(r=1,\cdots,N_d)$, we can see
\begin{align}\label{chap2-def-I-hatW-tildeW-hatw}
\tilde{R}_{\Gamma} \hat{u}
= \sum\limits_{r=1}^{N_d} \sum\limits_{k \in \mathcal{M}_F^{(r)}} T^{r} \hat{u}_{F_k,\Delta}
+ \sum\limits_{r=1}^{N_d} \sum\limits_{k \in \mathcal{M}_E^{(r)}} T^{r} \hat{u}_{E_k,\Delta}
+ \sum\limits_{k=1}^{N_F} \hat{u}_{F_k,\Pi}
+ \sum\limits_{k=1}^{N_E} \hat{u}_{E_k,\Pi}
+ \sum\limits_{k=1}^{N_V} \hat{u}_{V_k}.
\end{align}

Using the definitions \eqref{chap2-IGamma-1}, \eqref{chap2-IGamma-2}, \eqref{chap2-IGamma-3} and \eqref{def-tildeD} of the operator $I_{\Gamma}$ and $\tilde{D}$, the definitions of
the restriction operators $\tilde{R}_{\Delta}^{(r)}$ $(r=1,\cdots,N_d)$ and $\tilde{R}_{\Pi}$, and combining with \eqref{chap2-def-I-hatW-tildeW-hatw} and \eqref{chap2-decomposition-hatw}, it follows
\allowdisplaybreaks
\begin{align*} %
I_{\Gamma} \tilde{D} \tilde{R}_{\Gamma} \hat{{u}}
&= \sum\limits_{r=1}^{N_d} \sum\limits_{k \in \mathcal{M}_F^{(r)}} T_r D_{F_k}^{(r)} T^{r} \hat{u}_{F_k,\Delta}
 + \sum\limits_{r=1}^{N_d} \sum\limits_{k \in \mathcal{M}_E^{(r)}} T_r D_{E_k}^{(r)} T^{r} \hat{u}_{E_k,\Delta}
 + \sum\limits_{k=1}^{N_F} \hat{u}_{F_k,\Pi} + \sum\limits_{k=1}^{N_E} \hat{u}_{E_k,\Pi}+ \sum\limits_{k=1}^{N_V} \hat{u}_{V_k}\\\nonumber
&= \sum\limits_{k=1}^{N_F} \sum\limits_{\nu \in \mathcal{N}_{\mathcal{F}_k}} T_{\nu} D_{F_k}^{(\nu)} T^{\nu}  \hat{u}_{F_k,\Delta}
+ \sum\limits_{k=1}^{N_E} \sum\limits_{\nu \in \mathcal{N}_{\mathcal{E}_k}} T_{\nu} D_{E_k}^{(\nu)} T^{\nu}  \hat{u}_{E_k,\Delta}
+ \sum\limits_{k=1}^{N_F} \hat{u}_{F_k,\Pi} + \sum\limits_{k=1}^{N_E} \hat{u}_{E_k,\Pi} + \sum\limits_{k=1}^{N_V} \hat{u}_{V_k}\\\nonumber
&= \sum\limits_{k=1}^{N_F} \hat{u}_{F_k,\Delta} + \sum\limits_{k=1}^{N_E} \hat{u}_{E_k,\Delta} + \sum\limits_{k=1}^{N_F} \hat{u}_{F_k,\Pi} + \sum\limits_{k=1}^{N_E} \hat{u}_{E_k,\Pi} + \sum\limits_{k=1}^{N_V} \hat{u}_{V_k} = \hat{{u}},
\end{align*}

Noting that $\hat{u}$ is any given function of $\hat{W}$, then \eqref{chap2-lemma-R-hatW-tildeW-trans} holds.
\end{proof}

From lemma \ref{chap2-lemma-5-1}, and denote $\tilde{R}_{\Gamma} I_{\Gamma} \tilde{D}$ as the average operator $E_D$, we can see
\begin{eqnarray}\label{chap2-property-ED-5-2-1}
(E_D)^2 = E_D.
\end{eqnarray}

Using Lemma \ref{chap2-lemma-5-1} and noting that $\tilde{S}$ is Hermitian positive definite, similar to the proof of Lemma 3.4 in \cite{BrennerSung07:1429}, we can get
\begin{lemma}\label{chap2-theorem-5-1}
The minimum eigenvalue of the preconditioned operator $\hat{G}$ satisfies
\begin{align}\label{chap2-lambdamin-hatG-oper}
\lambda_{\min}(\hat{G}) \ge 1.
\end{align}
\end{lemma}

Next, we will derive the estimation of the maximum eigenvalue of $\hat{G}$.

We firstly introduce a jump operator $P_D: \tilde{W} \rightarrow \tilde{W}$,
which is a complementary projector of $E_D$ and satisfies
\begin{align}\label{chap2-def-Pd-0}
P_D = I - E_D,
\end{align}
where $I: \tilde{W} \rightarrow \tilde{W}$ is an identity operator.

Using \eqref{chap2-property-ED-5-2-1} and \eqref{chap2-def-Pd-0}, similar to the
estimation  of the maximum eigenvalue of $\hat{G}$ in the algebraic framework of \cite{KimChung15:571}, it follows
\begin{align}\label{chap2-transform-1}
\lambda_{\max}(\hat{G}) \le \lambda_{\max}(G_d),
\end{align}
where $G_d := (P_D)^H \tilde{S} P_D \tilde{S}^{-1}$.

For any given $\tilde{w} \in \tilde{W}$, using \eqref{chap2-ZJ-TildeW-FuncDecomp}, we have
\begin{align}\label{chap2-def-PD-tildew}
\tilde{w} = \sum\limits_{r=1}^{N_d} (\sum\limits_{k \in \mathcal{M}_F^{(r)}} w_{F_k,\Delta}^{(r)} + \sum\limits_{k \in \mathcal{M}_E^{(r)}} w_{E_k,\Delta}^{(r)}) + w_{\Pi},~~w_{\Pi} := \sum\limits_{k=1}^{N_F} w_{F_k,\Pi} + \sum\limits_{k=1}^{N_E} w_{E_k,\Pi} + \sum\limits_{k=1}^{N_V} w_{V_k},
\end{align}
where
\begin{align}\label{chap2-def-PD-tildew-para}
w_{X_k,\Delta}^{(r)} &= (\vec{w}_{X_k,\Delta}^{(r)})^T \Phi_{\Delta}^{X_k,r} \in W_{X_k,\Delta}^{(r)},~X = F, E, \\\label{chap2-def-PD-tildew-para-pi}
w_{X_k,\Pi} &= (\vec{w}_{X_k,\Pi})^T \Phi_{\Pi}^{X_k} \in W_{X_k,\Pi},~X = F, E,~
w_{V_k} = (\vec{w}_{V_k})^T \Phi_{V_k} \in W_{V_k},
\end{align}
here $\vec{w}_{X_k,\Delta}^{(r)} \in \mathbb{C}^{n_{\Delta}^{X_k}}$, $\vec{w}_{X_k,\Pi} \in \mathbb{C}^{n_{\Pi}^{X_k}} (X = F, E)$, and $\vec{w}_{V_k} \in \mathbb{C}^{n_{V_k}}$.

By using \eqref{chap2-transform-1}, noting that $G_d$ and $\tilde{S}^{-1} P_D^H \tilde{S} P_D$ have the same eigenvalue except 0,
and $\tilde{S}^{-1} P_D^H \tilde{S} P_D$ is symmetry associated with $\tilde{a}(\cdot,\cdot)$, we have
\begin{align*}
\lambda_{\max}(\hat{G}) \le \max\limits_{\tilde{w} \in  \tilde{W} \backslash \{0\}} \frac{\tilde{a}(\tilde{S}^{-1} P_D^H \tilde{S} P_D \tilde{w}, \tilde{w})}{\tilde{a}(\tilde{w}, \tilde{w})}.
\end{align*}
Further, using the definition \eqref{chap2-def-operator-tilde-S} of $\tilde{S}$, the following lemma holds.
\begin{lemma}\label{chap2-lemma-5-2}
The maximum eigenvalue of the preconditioned operator $\hat{G}$ satisfies
\begin{align}\label{chap2-5-26-proof-2-000}
\lambda_{\max}(\hat{G}) \le \max\limits_{\tilde{w} \in  \tilde{W} \backslash \{0\}} \frac{\tilde{a}(P_D \tilde{w}, P_D \tilde{w})}{\tilde{a}(\tilde{w}, \tilde{w})}.
\end{align}
\end{lemma}

In order to estimate the right hand of \eqref{chap2-5-26-proof-2-000}, we derive the expression of $P_D \tilde{w}$ for any $\tilde{w} \in \tilde{W}$ firstly.

\begin{lemma}\label{lemma-PD-express}
For any $\tilde{w} \in \tilde{W}$ defined in \eqref{chap2-def-PD-tildew}, we have
\begin{align}\label{chap2-exp-Pdwi}
P_D \tilde{W} = \sum\limits_{r=1}^{N_d} \sum\limits_{k \in \mathcal{M}_F^{(r)}}\sum\limits_{s \in \mathcal{N}_{\mathcal{F}_k} \backslash \{r\}}
(w_{D,\Delta}^{F_k,r,s} - \tilde{w}_{D,\Delta}^{F_k,r,s}) + \sum\limits_{r=1}^{N_d} \sum\limits_{k \in \mathcal{M}_E^{(r)}} \sum\limits_{s \in \mathcal{N}_{\mathcal{E}_k} \backslash \{r\}}(w_{D,\Delta}^{E_k,r,s} - \tilde{w}_{D,\Delta}^{E_k,r,s}),
\end{align}
where
\begin{align}\label{chap2-exp-Pdwi-parameter}
w_{D,\Delta}^{X_k,r,s}:= (\vec{w}_{X_k,\Delta}^{(r)})^T \Phi_{D,\Delta}^{X_k,r,s},~
\tilde{w}_{D,\Delta}^{X_k,r,s} := (\vec{w}_{X_k,\Delta}^{(s)})^T \Phi_{D,\Delta}^{X_k,r,s},~X = F, E,
\end{align}
here $\Phi_{D,\Delta}^{X_k,r,s} (X = F, E)$ are given by \eqref{chap2-ZJ-TDphik-ij-Def-00}.

\end{lemma}

\begin{proof}

By using \eqref{chap2-def-Pd-0}, $E_D = \tilde{R}_{\Gamma} I_{\Gamma} \tilde{D}$ and \eqref{def-tildeD}, we can rewrite $P_D \tilde{w}$ as
\begin{align}\label{chap2-lemmaLT-I-equ-proof-1-001-0}\nonumber
P_D \tilde{w} &= \tilde{w} - E_D \tilde{w}\\\nonumber
&=\tilde{w} - \tilde{R}_{\Gamma} I_{\Gamma} \tilde{D} \tilde{w}\\\nonumber
&= \sum\limits_{r=1}^{N_d} \sum\limits_{k \in \mathcal{M}_F^{(r)}} w_{F_k,\Delta}^{(r)}
+ \sum\limits_{r=1}^{N_d} \sum\limits_{k \in \mathcal{M}_E^{(r)}} w_{E_k,\Delta}^{(r)} + w_{\Pi} \\\nonumber
&~~- \tilde{R}_{\Gamma} I_{\Gamma} \sum\limits_{r=1}^{N_d} D_{\Delta}^{(r)} (\sum\limits_{k \in \mathcal{M}_F^{(r)}} w_{F_k,\Delta}^{(r)}
+ \sum\limits_{k \in \mathcal{M}_E^{(r)}} w_{E_k,\Delta}^{(r)}) -w_{\Pi} \\\nonumber
&= \sum\limits_{r=1}^{N_d} \sum\limits_{k \in \mathcal{M}_F^{(r)}} w_{F_k,\Delta}^{(r)}
+ \sum\limits_{r=1}^{N_d} \sum\limits_{k \in \mathcal{M}_E^{(r)}} w_{E_k,\Delta}^{(r)} \\
&~~- \tilde{R}_{\Gamma} I_{\Gamma} \sum\limits_{r=1}^{N_d} D_{\Delta}^{(r)} (\sum\limits_{k \in \mathcal{M}_F^{(r)}} w_{F_k,\Delta}^{(r)}
+ \sum\limits_{k \in \mathcal{M}_E^{(r)}} w_{E_k,\Delta}^{(r)}).
\end{align}

From \eqref{chap2-def-PD-tildew}, \eqref{chap2-def-R-property-1} and the definition of $\tilde{R}_{\Gamma}$,
 we can see that the third term of the right hand of \eqref{chap2-lemmaLT-I-equ-proof-1-001-0} satisfies
\allowdisplaybreaks
\begin{align}\label{chap2-equ-IR-1} \nonumber
&\tilde{R}_{\Gamma} I_{\Gamma} \sum\limits_{r=1}^{N_d} D_{\Delta}^{(r)} (\sum\limits_{k
\in \mathcal{M}_F^{(r)}} w_{F_k,\Delta}^{(r)} + \sum\limits_{k
\in \mathcal{M}_E^{(r)}} w_{E_k,\Delta}^{(r)})\\\nonumber
&= \tilde{R}_{\Gamma} (\sum\limits_{r=1}^{N_d} \sum\limits_{k \in \mathcal{M}_F^{(r)}}
T_r D_{F_k}^{(r)}
w_{F_k,\Delta}^{(r)} + \sum\limits_{r=1}^{N_d} \sum\limits_{k \in \mathcal{M}_E^{(r)}}
T_r D_{E_k}^{(r)}w_{E_k,\Delta}^{(r)})\\
&= \sum\limits_{r=1}^{N_d} \sum\limits_{k \in \mathcal{M}_F^{(r)}}
D_{F_k}^{(r)} \sum\limits_{s \in \mathcal{N}_{\mathcal{F}_k}} T_{W_{F_k}^{(r)}}^{W_{F_k}^{(s)}}
w_{F_k,\Delta}^{(r)} + \sum\limits_{r=1}^{N_d} \sum\limits_{k \in \mathcal{M}_F^{(r)}}
D_{E_k}^{(r)} \sum \limits_{s \in \mathcal{N}_{\mathcal{E}_k}} T_{W_{E_k}^{(r)}}^{W_{E_k}^{(s)}}
w_{E_k,\Delta}^{(r)},
\end{align}
where $T_{W_{X_k}^{(r)}}^{W_{X_k}^{(s)}} (r,s \in \mathcal{N}_{\mathcal{X}_k}, X = F,E)$ are the basis transformation operators from $W_{X_k}^{(r)}$ to $W_{X_k}^{(s)}$ which satisfy
\begin{align}\label{chap2-ZJ-Basis-WkiToWk-Tran-ij}
T_{W_{X_k}^{(r)}}^{W_{X_k}^{(s)}} \phi_l^{X_k,r} = \phi_l^{X_k,s},~l=1,\cdots,n_{X_k},~r,s \in \mathcal{N}_{\mathcal{X}_k}, X = F,E.
\end{align}

Combining \eqref{chap2-lemmaLT-I-equ-proof-1-001-0} with \eqref{chap2-equ-IR-1}, and using
~\eqref{chap2-def-D-trans}, \eqref{chap2-def-PD-tildew-para}, \eqref{chap2-ZJ-Basis-WkiToWk-Tran-ij},
\eqref{chap2-Zj-Def-Vec-TildPhiDel-1} and \eqref{chap2-ZJ-Dphik-ij-Def}, we obtain
\allowdisplaybreaks
\begin{align*}\nonumber
P_D \tilde{w}
&= \sum\limits_{r=1}^{N_d} \sum\limits_{k \in
\mathcal{M}_F^{(r)}} w_{F_k,\Delta}^{(r)}
+ \sum\limits_{r=1}^{N_d} \sum\limits_{k \in
\mathcal{M}_E^{(r)}} w_{E_k,\Delta}^{(r)}
 - \sum\limits_{r=1}^{N_d} \sum\limits_{k \in \mathcal{M}_F^{(r)}}
D_{F_k}^{(r)} \sum\limits_{\mu \in \mathcal{N}_{\mathcal{F}_k}} T_{W_{F_k}^{(r)}}^{W_{F_k}^{(\mu)}}
w_{F_k,\Delta}^{(r)} \\\nonumber
&~~-  \sum\limits_{r=1}^{N_d} \sum\limits_{k \in \mathcal{M}_E^{(r)}}
D_{E_k}^{(r)} \sum \limits_{\mu \in \mathcal{N}_{\mathcal{E}_k}} T_{W_{E_k}^{(r)}}^{W_{E_k}^{(\mu)}}
w_{E_k,\Delta}^{(r)}\\\nonumber
&= \sum\limits_{r=1}^{N_d} \sum\limits_{k \in \mathcal{M}_F^{(r)}} \sum\limits_{\mu \in \mathcal{N}_{\mathcal{F}_k}}
D_{F_k}^{(\mu)} w_{F_k,\Delta}^{(r)}
   - \sum\limits_{r=1}^{N_d} \sum\limits_{k \in \mathcal{M}_F^{(r)}} D_{F_k}^{(r)} \sum\limits_{\mu \in \mathcal{N}_{\mathcal{F}_k}} T_{W_{F_k}^{(r)}}^{W_{F_k}^{(\mu)}} w_{F_k,\Delta}^{(r)}\\\nonumber
&~~+\sum\limits_{r=1}^{N_d} \sum\limits_{k \in \mathcal{M}_E^{(r)}} \sum\limits_{\mu \in \mathcal{N}_{\mathcal{E}_k}} D_{E_k}^{(\mu)} w_{E_k,\Delta}^{(r)}
 - \sum\limits_{r=1}^{N_d} \sum\limits_{k \in \mathcal{M}_E^{(r)}} D_{E_k}^{(r)} \sum\limits_{\mu \in \mathcal{N}_{\mathcal{E}_k}}T_{W_{E_k}^{(r)}}^{W_{E_k}^{(\mu)}} w_{E_k,\Delta}^{(r)}\\\nonumber
&= \sum\limits_{k=1}^{N_F} \sum\limits_{r \in \mathcal{N}_{\mathcal{F}_k}} \sum\limits_{\mu \in \mathcal{N}_{\mathcal{F}_k} \backslash \{r\}} \left(D_{F_k}^{(\mu)}
w_{F_k,\Delta}^{(r)} - D_{F_k}^{(r)} T_{W_{F_k}^{(r)}}^{W_{F_k}^{(\mu)}}
w_{F_k,\Delta}^{(r)}\right)
\\\nonumber
&~~+\sum\limits_{k=1}^{N_E} \sum\limits_{r \in \mathcal{N}_{\mathcal{E}_k}}
\sum\limits_{\mu \in \mathcal{N}_{\mathcal{E}_k} \backslash \{r\} } (D_{E_k}^{(\mu)}w_{E_k,\Delta}^{(r)} - D_{E_k}^{(r)} T_{W_{E_k}^{(r)}}^{W_{E_k}^{(\mu)}} w_{E_k,\Delta}^{(r)})\\\nonumber
&= \sum\limits_{k=1}^{N_F} \sum\limits_{r \in \mathcal{N}_{\mathcal{F}_k}} \sum\limits_{\mu \in \mathcal{N}_{\mathcal{F}_k} \backslash \{r\}} \left(D_{F_k}^{(\mu)}
w_{F_k,\Delta}^{(r)} - D_{F_k}^{(\mu)} T_{W_{F_k}^{(\mu)}}^{W_{F_k}^{(r)}}
w_{F_k,\Delta}^{(\mu)}\right) \\\nonumber
&~~+ \sum\limits_{k=1}^{N_E} \sum\limits_{r \in \mathcal{N}_{\mathcal{E}_k}}
\sum\limits_{\mu \in \mathcal{N}_{\mathcal{E}_k} \backslash \{r\} } \left(D_{E_k}^{(\mu)}w_{E_k,\Delta}^{(r)} - D_{E_k}^{(\mu)} T_{W_{E_k}^{(\mu)}}^{W_{E_k}^{(r)}} w_{E_k,\Delta}^{(\mu)}\right)
\\
&= \sum\limits_{r=1}^{N_d} \sum\limits_{k \in \mathcal{M}_F^{(r)}}\sum\limits_{s \in \mathcal{N}_{\mathcal{F}_k} \backslash \{r\}}
\left(w_{D,\Delta}^{F_k,r,s} - \tilde{w}_{D,\Delta}^{F_k,r,s}\right) + \sum\limits_{r=1}^{N_d} \sum\limits_{k \in \mathcal{M}_E^{(r)}} \sum\limits_{s \in \mathcal{N}_{\mathcal{E}_k} \backslash \{r\}}\left(w_{D,\Delta}^{E_k,r,s} - \tilde{w}_{D,\Delta}^{E_k,r,s} \right).
\end{align*}

The proof of \eqref{chap2-exp-Pdwi} has been completed.
\end{proof}

Combing Lemma \ref{chap2-lemma-5-2} and Lemma \ref{lemma-PD-express}, the following lemma holds.
\begin{lemma}\label{chap2-lemma-max}
For any given thresholds $\Theta_E, \Theta_F \ge 1$, the maximum eigenvalue of the adaptive BDDC preconditioned operator $\hat{G}$ satisfies
\begin{align}\label{chap2-eig-max}
\lambda_{\max}(\hat{G}) \le C \Theta,
\end{align}
where $\Theta = \max\{\Theta_E, \Theta_F\}$, $C = 2 C_{FE}^2$, here $C_{FE}$ is a constant depending only on the number of common faces and edges per subdomain and the number of subdomains sharing an edge.
\end{lemma}

\begin{proof}
In fact, if we can prove
\begin{align*}
\max\limits_{\tilde{w} \in  \tilde{W} \backslash \{0\}} \frac{\tilde{a}(P_D \tilde{w}, P_D \tilde{w})}{\tilde{a}(\tilde{w}, \tilde{w})} \le C \Theta.
\end{align*}
then \eqref{chap2-eig-max} holds.

Using
~\eqref{chap2-def-operator-tilde-A},
~\eqref{chap2-def-PD-tildew}, \eqref{chap2-exp-Pdwi} and \eqref{def-normAi}
, we can see the above inequality is equivalent to
\begin{align}\label{chap2-Pb-max-eigenvalue-estimate-equiv-1-000-a}
\sum\limits_{r=1}^{N_d} |(P_D \tilde{w})^{(r)}|_{a_r}^2 \le C
\Theta \sum\limits_{r=1}^{N_d} |\tilde{w}^{(r)}|_{a_r}^2,~\forall \tilde{w} \in \tilde{W} \backslash \{0\},
\end{align}
where
\begin{align}\label{chap2-PDi}
&\tilde{w}^{(r)} = \sum\limits_{k \in \mathcal{M}_F^{(r)}} (w_{F_k,\Delta}^{(r)} + w_{F_k,\Pi}^{(r)})
+ \sum\limits_{k \in \mathcal{M}_E^{(r)}} (w_{E_k,\Delta}^{(r)} + w_{E_k,\Pi}^{(r)})
+ \sum\limits_{k\in \mathcal{M}_V^{(r)}} w_{V_k}^{(r)},\\\label{chap2-PDi}
&(P_D \tilde{w})^{(r)} =  \sum\limits_{k \in \mathcal{M}_F^{(r)}} \sum\limits_{s \in \mathcal{N}_{\mathcal{F}_k}\backslash \{r\}} (w_{D,\Delta}^{F_k,r,s} - \tilde{w}_{D,\Delta}^{F_k,r,s}) + \sum\limits_{k \in \mathcal{M}_E^{(r)}} \sum\limits_{s \in \mathcal{N}_{\mathcal{E}_k}\backslash \{r\}} (w_{D,\Delta}^{E_k,r,s} - \tilde{w}_{D,\Delta}^{E_k,r,s}),
\end{align}
here $w_{X_k,\Pi}^{(r)} = (\vec{w}_{X_k,\Pi})^T \Phi_{\Pi}^{X_k,r} \in W_{X_k,\Pi}^{(r)} (X = F, E)$, $w_{V_k}^{(r)} = (\vec{w}_{V_k})^T \Phi^{V_k,r} \in W_{V_k}^{(r)}$, and $w_{D,\Delta}^{X_k,r,s}, \tilde{w}_{D,\Delta}^{X_k,r,s}$ $(X = F, E)$ are defined in \eqref{chap2-exp-Pdwi-parameter}.

By using \eqref{chap2-PDi}, \eqref{chap2-ZJ-Tphik-TDphik-Prop}, \eqref{chap2-4-1-2-tildew}
and the inequality $|\sum\limits_{l=1}^J \balpha_l|_{a_r}^2 \le J \sum\limits_{l=1}^J |\balpha_l|_{a_r}^2$, we obtain
\begin{align}\label{chap2-final-ineq-1}\nonumber
\sum\limits_{r=1}^{N_d} \big{|}(P_D\tilde{w})^{(r)}|_{a_r}^2
&\le 2 C_{FE}
\sum\limits_{r=1}^{N_d}\left(\sum\limits_{k \in
\mathcal{M}_F^{(r)}} \sum\limits_{s \in \mathcal{N}_{\mathcal{F}_k} \backslash \{r\}} \left(|w_{D,\Delta}^{F_k,r,s}|_{a_r}^2 +
|\tilde{w}_{D,\Delta}^{F_k,r,s}|_{a_r}^2\right) \right.\\\nonumber
&~~~\left.+ \sum\limits_{k \in
\mathcal{M}_E^{(r)}} \sum\limits_{s \in \mathcal{N}_{\mathcal{E}_k} \backslash \{r\}}\left(|w_{D,\Delta}^{E_k,r,s}|_{a_r}^2 +
|\tilde{w}_{D,\Delta}^{E_k,r,s}|_{a_r}^2\right)\right)
\\\nonumber
&= 2 C_{FE} \left([\sum\limits_{r=1}^{N_d} \sum\limits_{k \in
\mathcal{M}_F^{(r)}} \sum\limits_{s \in \mathcal{N}_{\mathcal{F}_k} \backslash \{r\}} |w_{D,\Delta}^{F_k,r,s}|_{a_r}^2 +
\sum\limits_{r=1}^{N_d} \sum\limits_{k \in \mathcal{M}_F^{(r)}} \sum\limits_{s \in \mathcal{N}_{\mathcal{F}_k} \backslash \{r\}}
|\tilde{w}_{D,\Delta}^{F_k,s,r}|_{a_s}^2] \right. \\\nonumber
&~~~~\left. +
[\sum\limits_{r=1}^{N_d} \sum\limits_{k \in
\mathcal{M}_E^{(r)}} \sum\limits_{s \in \mathcal{N}_{\mathcal{E}_k} \backslash \{r\}} |w_{D,\Delta}^{E_k,r,s}|_{a_r}^2 +
\sum\limits_{r=1}^{N_d} \sum\limits_{k \in \mathcal{M}_E^{(r)}} \sum\limits_{s \in \mathcal{N}_{\mathcal{E}_k} \backslash \{r\}}
|\tilde{w}_{D,\Delta}^{E_k,s,r}|_{a_s}^2]\right)\\\nonumber
&= 2 C_{FE} \left(\sum\limits_{r=1}^{N_d} \sum\limits_{k \in \mathcal{M}_F^{(r)}} \sum\limits_{s \in \mathcal{N}_{\mathcal{F}_k} \backslash \{r\}}
\left(|w_{D,\Delta}^{F_k,r,s}|_{a_r}^2 +
|\tilde{w}_{D,\Delta}^{F_k,s,r}|_{a_s}^2\right)\right.\\ \nonumber
&\left.~~+ \sum\limits_{r=1}^{N_d} \sum\limits_{k \in
\mathcal{M}_E^{(r)}} \sum\limits_{s \in \mathcal{N}_{\mathcal{E}_k} \backslash \{r\}} \left(|w_{D,\Delta}^{E_k,r,s}|_{a_r}^2 +
|\tilde{w}_{D,\Delta}^{E_k,s,r}|_{a_s}^2\right)\right)\\\nonumber
&\lesssim 2 C_{FE} \left(\Theta_F \sum\limits_{r=1}^{N_d} \sum\limits_{k \in
\mathcal{M}_F^{(r)}} |\bar{w}_{F_k,\Delta}^{(r)} +
\bar{w}_{F_k,\Pi}^{(r)}|_{a_r}^2 \right.
+ \left. \Theta_E \sum\limits_{r=1}^{N_d} \sum\limits_{k \in
\mathcal{M}_E^{(r)}} |\bar{w}_{E_k,\Delta}^{(r)} +
\bar{w}_{E_k,\Pi}^{(r)}|_{a_r}^2 \right)
\\
&\le 2 C_{FE}^2 \Theta \sum\limits_{r=1}^{N_d}
|\tilde{w}^{(r)}|_{a_r}^2,
\end{align}
where
\begin{align*}
\tilde{w}_{D,\Delta}^{X_k,s,r} := (\vec{w}_{X_k,\Delta}^{(r)})^T \Phi_{D,\Delta}^{X_k,s,r},~
\bar{w}_{X_k,\zeta}^{(r)} = (\vec{w}_{X_k,\zeta}^{(r)})^T \bar{\Phi}_{\zeta}^{X_k,r} \in \bar{W}_{X_k,\zeta}^{(r)}, ~\zeta = \Delta,\Pi,~X = F, E.
\end{align*}

Finally, \eqref{chap2-Pb-max-eigenvalue-estimate-equiv-1-000-a} follows from  \eqref{chap2-final-ineq-1}.
\end{proof}

By Lemma \ref{chap2-theorem-5-1} and Lemma \ref{chap2-lemma-max}, the main result of this section holds.
\begin{theorem}\label{chap2-condition-number}
For any given thresholds $\Theta_{E}, \Theta_{F} \ge 1$, the condition number of the adaptive BDDC preconditioned operator $\hat{G}$ satisfies
\begin{align*}
\kappa(\hat{G}) \le C \Theta,
\end{align*}
where $\Theta =\max\{\Theta_{E}, \Theta_{F}\}$, $C$ is a constant depending only on the number of common faces and common edges per subdomain and the number of subdomains sharing a common edge.
\end{theorem}

\section{Numerical experiments}\setcounter{equation}{0}

In this section, numerical experiments are presented for solving Helmholtz equations in three dimensions.
Since the stiffness matrix of the PWLS system \eqref{chap2-ZJ-Schur-System} is Hermitian and positive definite, we will apply preconditioned conjugate gradient (PCG) algorithm to solve \eqref{chap2-ZJ-Schur-System},
and the iteration is stopped either the relative residual is less than $10^{-5}$ or the iteration counts are greater
than $100$. These algorithms are tested in a machine with Intel(R) Xeon(R) CPU E5-2650 v2 2.60 GHz and 96-GB memory.

In the following experiments, we select a benchmark problem to study the properties of the adaptive BDDC method.
We choose $\Omega$ as a unit cube, and adopt a uniform partition $\mathcal{T}_h$
for the domain as follows: $\Omega$ is divided into some cubic elements with the same size $h$,
where $h$ denotes the length of the longest edge of the elements.
In the following tables, $n$ denotes the number of subdomains in each direction, $m$ denotes
the number of complete elements in each direction of one subdomain, $p$ is the number of plane waves used in each element, Iter is the number of iterations needed in the PCG algorithm, $\lambda_{\min}$ and
$\lambda_{\max}$ separately denote the minimum and maximum eigenvalues of the preconditioned system,
cond is the condition number of the preconditioned system,
pnum is the number of total primal unknowns,
pnumF and pnumE are separately the number of primal unknowns on faces and edges, the average number of primal unknowns in each face or each edge are given in the parentheses. M1 and M2 separately denote the adaptive BDDC algorithm with deluxe scaling matrices \cite{DohrmannPechstein2012} and multiplicity scaling matrices \cite{RixenFarhat99:489}.

\begin{example}
\begin{equation}\label{example-1}
\left\{\begin{array}{ll}
-\Delta u - \kappa^2 u = 0, & \mbox{in}~\Omega,\\
\frac{\partial u}{\partial {\bf n}} + i \kappa u = g, & \mbox{on}~\partial \Omega,
\end{array}\right.
\end{equation}
where $\Omega = (0,1)\times(0,1)\times(0,1)$, and $g = i\kappa(1+{\bf v} \cdot {\bf n}e^{i\kappa {\bf v}_0 \cdot {\bf x}})$.

The analytic solution of the problem can be obtained in the close form as
$$u_{ex} ({\bf x}) = e^{i\kappa {\bf v}_0 \cdot {\bf x}},$$
where ${\bf v}_1 = (\tan(-\pi/10), 0, \tan(\pi/5))^T$, ${\bf v}_0 = {\bf v}_1/\|{\bf v}_1\|_2$.
\end{example}

%

In Table \ref{ex-table-1} and Table \ref{ex-table-1-1}, we study the convergence behavior and the changes in the scale of the coarse spaces with respect to the condition number indicators $\Theta_{F}$ and $\Theta_{E}$.

We set $\kappa = 8\pi$, $p = 18$, $n(m)=3(3)$, Table \ref{ex-table-1} summarizes the results of the two-level adaptive BDDC method
with economic (see \cite{KlawonnRadtke16:75}, where $\eta = h$) and noneconomic generalized eigenvalue problems for different choice of $\Theta_F$ and $\Theta_E$.

\begin{table}[H]
{\footnotesize\centering
\caption{The results for different choice of the tolerances $\Theta_F$ and $\Theta_E$ with economic-version and noneconomic-version two-level adaptive BDDC method.}
\label{ex-table-1}
\vskip 0.1cm
\begin{tabular}{{|c|c|c|c|c|c|c|c|c|c|}}\hline
                             &\multirow{2}{*}{method}     & \multicolumn{4}{c|}{noneconomic-version}   & \multicolumn{4}{c|}{economic-version($\eta = h$)}        \\\cline{3-10}
                             &       &Iter   & pnum   & pnumF   & pnumE        & Iter & pnum   & pnumF       &pnumE         \\\hline
$\Theta_F = 1 + log(m)$       &M1  & 7   & 3626   &1538  &1944 & 7    &3646    &1558  &1944   \\\cline{2-10}
$\Theta_E = 1 + log(m)$       &M2  & 9   & 8534   &6446  &1944 & 9    &8530    &6442  &1944   \\\hline
$\Theta_F = 1 + log(m)$       &M1  & 7   & 3526   &1538  &1844 & 7    &3542    &1558  &1840   \\\cline{2-10}
$\Theta_E = 4m$               &M2  & 11  & 8406   &6446  &1816 & 11   &8398    &6442  &1812   \\\hline
$\Theta_F = 1 + log(m)$       &M1  & 16  & 2086   &1538  &404  & 15   &1986    &1558  &284   \\\cline{2-10}
$\Theta_E = 10^3$             &M2  & 80  & 7070   &6446  &480  & 89   &6930    &6442  &344   \\\hline
$\Theta_F = 1 + log(m)$       &M1  & 16  & 2074   &1538  &392  & 16   &1946    &1558  &244   \\\cline{2-10}
$\Theta_E = 10^3log(m)$       &M2  & 102 & 7042   &6446  &452  & 111  &6890    &6442  &304   \\\hline
$\Theta_F = 4m$                &M1  & 24  & 1090   &542   &404  & 24   &974     &546   &284   \\\cline{2-10}
$\Theta_E = 10^3$              &M2  & 87  & 3264   &2640  &480  & 97   &3138    &2650  &344        \\\hline
\end{tabular}
}
\end{table}

From this table, we can see that for $\Theta_E = 1+log(m)$,
the dofs on edges are totally selected as the primal unknows,
and as $\Theta_E$ increase from $1+log(m)$ to $10^3$, the number of primal unknows on edges significantly reduced.
Similarly, the number of primal unknows on faces reduced as $\Theta_F$ increases.
We also note that, compared with M2, the iteration counts of M1 are not large and increase slowly when the condition number indicators (or tolerances) increase.
It worth pointing out that, when we use the economic-version BDDC algorithm, not only the iteration counts of the PCG algorithm dose not increases,
but in most cases the number of the primal unknowns has slightly decreased.

We set $p = 18$, $n(m)=3(3)$ and adopt deluxe scaling in our proposed adaptive BDDC algorithm with economic generalized eigenvalue problems, Table \ref{ex-table-1-1} shows the numerical results for different wave numbers and different condition number indicators $\Theta_{F}$ and $\Theta_{E}$.
\begin{table}[H]
\centering\caption{The results for different $\Theta_F, \Theta_E$ and different wave numbers.}
\label{ex-table-1-1}
\vskip 0.1cm
\begin{tabular}{{|c|c|c|c|c|c|c|c|c|}}\hline
\multirow{2}{*}{$\Theta_F, \Theta_E$}  &\multicolumn{4}{c|}{$8\pi$}     &\multicolumn{4}{c|}{$16\pi$} \\\cline{2-9}
                                      & pnumF  &pnumE  &cond  &Iter    & pnumF  &pnumE  &cond  &Iter \\\hline
1                                     & 7662   &1944   &1.03  &3       & 7976   &1944   &1.03  &3    \\
10                                    & 772    &1888   &4.12  &13      & 440    &1836   &2.39  &12   \\
$10^2$                                & 292    &1248   &7.06  &18      & 260    &640    &14.70 &23   \\
$10^3$                                & 260    &284    &24.09 &33      & 260    &224    &22.60 &26   \\\hline
\end{tabular}
\end{table}
As we can see in Table \ref{ex-table-1-1}, the number of primal unknowns on faces and edges decreases as the condition number indicators $\Theta_F, \Theta_E$ increase,
the condition number of the preconditioned system and the total number of PCG iterations increase as the condition number indicators  $\Theta_F, \Theta_E$ increase.

In the following experiments, we choose $\Theta_F = 4m, \Theta_E = 1000$ and adopt the economic-version.

To measure the accuracy of the numerical solution, we introduce a relative $L^2$-error as:
\begin{align*}
err = \frac{\|u_{ex} - u_{h}\|_{L^2(\Omega)}}{\|u_{ex}\|_{L^2(\Omega)}},
\end{align*}
and as the PWLS method has the ``wave number polution" phenomenon, we keep $p=28$ and decrease $h$ to control the relative error less than $10^{-2}$ in the next experiments.

\begin{table}[H]
{\footnotesize \centering\caption{The efficiency of the two-level adaptive BDDC algorithm with variable wave number $\kappa$}
\label{ex-table-2}\vskip 0.1cm
\begin{tabular}{{|c|c|c|c|c|c|c|c|c|c|}}\hline
$\kappa$                 &$N_h$      & err   & method  & $\lambda_{\min}$ & $\lambda_{\max}$ &pnum & pnumF  &pnumE  & Iter \\\hline
\multirow{2}{*}{$8\pi$}  &\multirow{2}{*}{$11^3$}  &\multirow{2}{*}{9.68E-03}&M1 &1.00  &6.07   &7644  &3873(26.89)	&3015(27.92)  &16  \\
                         &                         &                               &M2 &1.00  &90.41  &14571 &10251(71.18)	&3564(33.00)  &65\\\hline
\multirow{2}{*}{$10\pi$} &\multirow{2}{*}{$15^3$} &\multirow{2}{*}{8.23E-03} &M1 &1.00  &8.46   &10365 &5406(37.54)	&4203(38.92)  &19 \\
                         &                         &                               &M2 &1.00  &113.62 &25833 &19911(138.27)	&5166(47.83)  &73\\\hline
\multirow{2}{*}{$16\pi$} &\multirow{2}{*}{$27^3$} &\multirow{2}{*}{9.19E-03} &M1 &1.00  &12.13  &19236 &10686(74.21)	&7794(72.17)  &23  \\
                         &                         &                               &M2 &-     &-      &61952 &51260(355.97)	&9936(92.00)  &- \\\hline
\end{tabular}
}
\end{table}

The results listed in Table \ref{ex-table-2} show that the iteration counts are not large and increase slowly as the wave number increases,
i.e., the growth rates of the iteration numbers are much smaller than the growth rates of the scales of the discrete systems.
But the cost of keeping the number of iterations is that the size of the coarse space increases sharply with the increase of wave number.
Especially, when $\kappa = 16\pi$, due to the limitation of the computer's memory, the program of the two-level adaptive BDDC algorithm with multiplicity scaling matrices can not be calculated properly. From this point of view, it is better to adopt deluxe scaling matrices in the adaptive BDDC algorithm for the Helmholtz problem with large wave numbers.

Further, we also show the efficiency of the two-level adaptive BDDC algorithm with deluxe scaling matrices for a fixed number of subdomains and a fixed number of complete elements in each direction of one subdomain respectively.

\begin{table}[H]
\centering\caption{
The efficiency of the two-level adaptive BDDC algorithm with variable number of complete elements $m$ in each direction of one subdomain.
}
\label{ex-table-3}\vskip 0.1cm
\begin{tabular}{{|c|c|c|c|c|c|c|c|c|c|}}\hline
$m$	& pnum	  & pnumF	    & pnumE	    &$\lambda_{\min}$ & $\lambda_{\max}$	&Iter \\\hline
2	&7644	  &3873(26.89)	&3015(27.92)	&1.00             &6.07	                &16   \\
3	&10812	  &5682(39.46)	&4374(40.50)	&1.00             &8.72	                &19   \\
4	&14156	  &7694(53.43)	&5706(52.83)	&1.00             &8.80	                &19   \\\hline
\end{tabular}
\end{table}

\begin{table}[H]
\centering\caption{
The efficiency of the two-level adaptive BDDC algorithm with variable number of subdomains $n$ in each direction.
}
\label{ex-table-4}\vskip 0.1cm
\begin{tabular}{{|c|c|c|c|c|c|c|c|c|c|}}\hline
$n$	&pnum	  &pnumF	    &pnumE	    &$\lambda_{\min}$ & $\lambda_{\max}$	&Iter \\\hline
4	&3030     &1680(11.67)  &864(8.00)   	&1.00             &8.78	                &20   \\
5	&7732     &4228(14.09)  &2352(9.80)  	&1.00             &10.77	            &22   \\
6	&16010    &8460(15.67)  &5300(11.78)	&1.00             &11.95                &22   \\
7	&28470    &15114(17.14) &9468(12.52)	&1.00             &15.46	            &24   \\\hline
\end{tabular}
\end{table}

In Table \ref{ex-table-3}, we set the wave number $\kappa = 8\pi$, $p=28$ and the number of subdomains in each direction $n=4$.
The results show that the iteration numbers are mildly dependent on the mesh size,
the minimum eigenvalues of the preconditioned systems are larger than 1, and the maximum eigenvalues are mildly dependent on the mesh size.
In Table \ref{ex-table-4}, we set the wave number $\kappa = 8\pi$, $p=18$ and the number of complete elements $m=2$ in each direction of one subdomain, and we have the same experiment results, but it is worth pointing out that the number of the primal unknowns have
significantly increased as the number of subdomains increase.
Since the coarse matrix in this algorithm is complex and dense, and the direct solver is applied to the coarse problem,
it brings great challenges to the computer hardware, and needs more time and memory cost.

In the next experiments, we would like to test the efficiency of our multi-level adaptive BDDC algorithm with deluxe scaling
and economic generalized eigenvalue problems, where the wave number $\kappa = 8\pi$, $m=2$, the PCG algorithm stopped either the iteration counts are greater than 100 or the relative residual is reduced by the factor of $10^{-5}$ at level $0$ and $10^{-2}$ at other levels, and four
subdomains at the finer level are treated as a coarser subdomain.

\begin{table}[H]
{\footnotesize
\centering\caption{The efficiency of the multi-level adaptive BDDC algorithm with variable number of subdomains
}
\label{ex-table-5}\vskip 0.1cm
\begin{tabular}{{|c|c|c|c|c|c|c|c|c|c|}}\hline
$n$     &pnum	  &pnumF	    &pnumE	    &$\lambda_{\min}$ & $\lambda_{\max}$	&Iter\\\hline
3 level($p=18$)      &             &             &               &                 &                     &     \\
4/2     &3030/138          &1680/96      &864/24         &1.00             &8.78                 &20   \\
6/3     &16010/1574        &8460/1002    &5300/428       &1.00             &11.95                &22   \\
8/4	    &46746/5389	       &24696/3354   &15876/1549     &1.00             &17.08                &26   \\
10/5    &104418/13280      &54036/8176   &37260/3952     &1.00             &14.70                &25   \\\hline
4 level($p=15$)      &             &             &               &                 &                     &     \\
8/4/2   &34888/4411/309    &19110/2871/273   &10633/1135/21  &1.00         &25.92                &33   \\
12/6/3  &139282/19525/2324 &75757/12070/1836 &43560/5580/368 &1.00         &17.93                &28   \\\hline
\end{tabular}
}
\end{table}

The results are listed in Table \ref{ex-table-5}, especially, we list the number of subdomains in each direction ($n$), the total number of primal unknows (pnum), the total number of primal unknows on faces (pnumF), the total number of primal unknows on edges (pnumE) at each level. From this table, we can see that the PCG iteration number (or the condition number) slightly increases when using more levels (see the results of $n=8$ in the finest level) or increasing the number of subdomains in the finest level.
And in addition, it is worth point out that compared with the finest level, the number of dofs at the coarsest level is reduced significantly,
i.e., the multilevel algorithm is
effective for reducing the number of dofs at the coarse problem, and can be used to solve large wave number
problems efficiently.

\section{Conclusions}

In this paper,
by introducing some auxiliary spaces, dual-primal basis functions, and operators with essential properties, BDDC algorithms with adaptive primal unknowns are developed and analyzed for the PWLS discretizations of the three-dimensional
Helmholtz equations.
Since the dofs of the PWLS discretization are defined on
elements rather than vertices or edges, we introduce a special ``interface" and the corresponding sesquilinear form for
each subdomain.
The coarse components are obtained by solving two types of local generalized
eigenvalue problems for each common face and each common edge.
We prove that the
condition number of the two-level adaptive BDDC preconditioned system is bounded above by $C \Theta$,
where $C$ is a constant depending only on the number of common faces and common edges per subdomain and the number of subdomains sharing a common edge,
$\Theta$ is the maximum of $\Theta_F$ and $\Theta_E$.
Some technical approaches are proposed to improve the computing efficiency, such as choosing the appropriate threshold and adopting the economic generalized eigenvalue problems,
and multilevel algorithm is designed to resolve
the bottleneck of large scale coarse problem.
Numerical results are presented to verify the robustness and efficiency of
the proposed approaches.
Further, we will devote to the parallel implementation of our multi-level adaptive BDDC algorithms, and extend this algorithms to the non-homogeneous Helmholtz equations with constant or variable wave numbers.

\section*{Acknowledgements}

The authors wish to thank the anonymous referees for many
insightful comments which led to great improvement in the results and the presentation of the paper.

This work is supported by the
National Natural Science Foundation of China (Grant Nos. 11971414, 11671159,11201398),
Scientific Research Fund of Hunan Provincial Education Department (Grant No. 18B082),
Natural Science Foundation of Guangdong Province (Grant No. 2016A030313842),
Project funded by China Postdoctoral Science Foundation (Grant No. 2019M652925),
Characteristic Innovation Projects of Guangdong colleges and universities (Grant No. 2018KTSCX044) and General Project topic of Science and Technology in Guangzhou, China (Grant No. 201904010117).


\bibliographystyle{abbrv}

\begin{thebibliography}{10}
\expandafter\ifx\csname url\endcsname\relax
  \def\url#1{\texttt{#1}}\fi
\expandafter\ifx\csname urlprefix\endcsname\relax\def\urlprefix{URL }\fi
\expandafter\ifx\csname href\endcsname\relax
  \def\href#1#2{#2} \def\path#1{#1}\fi

\bibitem{HiptmairMoiola16:237}
R.~Hiptmair, A.~Moiola, I.~Perugia, A survey of {Trefftz} methods for the
  {Helmholtz} equation, in: Building bridges: connections and challenges in
  modern approaches to numerical partial differential equations, Springer,
  2016, pp. 237--279.

\bibitem{CessenatDespres98:255}
O.~Cessenat, B.~Despres, Application of an ultra weak variational formulation
  of elliptic {PDEs} to the two-dimensional {Helmholtz} problem, SIAM J. Numer.
  Anal. 35~(1) (1998) 255--299.

\bibitem{GamalloAstley07:406}
P.~Gamallo, R.~J. Astley, A comparison of two {Trefftz-type} methods: the
  ultraweak variational formulation and the least-squares method, for solving
  shortwave {2-D Helmholtz} problems, Int. J. Numer. Methods Eng. 71~(4) (2007)
  406--432.

\bibitem{KovalevskyLadeveze12:142}
L.~Kovalevsky, P.~Ladev{\'e}ze, H.~Riou, The {Fourier} version of the
  variational theory of complex rays for medium-frequency acoustics, Comput.
  Methods Appl. Mech. Eng. 225 (2012) 142--153.

\bibitem{LadevezeArnaud01:193}
P.~Ladev{\'e}ze, L.~Arnaud, P.~Rouch, C.~Blanz{\'e}, The variational theory of
  complex rays for the calculation of medium-frequency vibrations, Eng.
  Computation. 18~(1/2) (2001) 193--214.

\bibitem{MonkWang99:121}
P.~Monk, D.~Wang, A least-squares method for the {Helmholtz} equation, Comput.
  Methods Appl. Mech. Eng. 175~(1) (1999) 121--136.

\bibitem{HuYuan14:587}
Q.~Hu, L.~Yuan, A weighted variational formulation based on plane wave basis
  for discretization of {Helmholtz} equations, Int. J. Numer. Anal. Mod. 11~(3)
  (2014) 587--607.

\bibitem{HuYuan18:245}
Q.~Hu, L.~Yuan, A plane wave method combined with local spectral elements for
  nonhomogeneous {Helmholtz} equation and time-harmonic {Maxwell} equations,
  Adv. Comput. Math. 44~(1) (2018) 245--275.

\bibitem{GittelsonHiptmair09:297}
C.~J. Gittelson, R.~Hiptmair, I.~Perugia, Plane wave discontinuous {Galerkin}
  methods: analysis of the h-version, M2AN Math. Model. Numer. Anal. 43~(2)
  (2009) 297--331.

\bibitem{HiptmairMoiola11:264}
R.~Hiptmair, A.~Moiola, I.~Perugia, Plane wave discontinuous {Galerkin} methods
  for the {2D Helmholtz} equation: analysis of the p-version, SIAM J. Numer.
  Anal. 49~(1) (2011) 264--284.

\bibitem{ToselliWidlund05Book}
A.~Toselli, O.~B. Widlund, Domain decomposition methods: algorithms and theory,
  Vol.~34, Springer Series in Computational Mathematics, 2005.

\bibitem{FarhatMacedo99:231}
C.~Farhat, A.~Macedo, R.~Tezaur, {FETI-H}: {A} scalable domain decomposition
  method for high frequency exterior {Helmholtz} problems, in: Eleventh
  International Conference on Domain Decomposition Method, 1999, pp. 231--241.

\bibitem{FarhatAvery05:499}
C.~Farhat, P.~Avery, R.~Tezaur, J.~Li, {FETI-DPH}: a dual-primal domain
  decomposition method for acoustic scattering, J. Comput. Acoust. 13~(03)
  (2005) 499--524.

\bibitem{GanderMagoules02:38}
M.~J. Gander, F.~Magoules, F.~Nataf, Optimized {Schwarz} methods without
  overlap for the {Helmholtz} equation, SIAM J. Sci. Comput. 24~(1) (2002)
  38--60.

\bibitem{GanderHalpern07:163}
M.~J. Gander, L.~Halpern, F.~Magoules, An optimized {Schwarz} method with
  two-sided {Robin} transmission conditions for the {Helmholtz} equation, Int.
  J. Numer. Methods Fluids 55~(2) (2007) 163--175.

\bibitem{ChenLiu16:921}
W.~Chen, Y.~Liu, X.~Xu, A robust domain decomposition method for the
  {Helmholtz} equation with high wave number, ESAIM: Math. Model. Num. 50~(3)
  (2016) 921--944.

\bibitem{EngquistYing11:686}
B.~Engquist, L.~Ying, Sweeping preconditioner for the {Helmholtz} equation:
  moving perfectly matched layers, Multiscale Model. Simul. 9~(2) (2011)
  686--710.

\bibitem{ChenXia13:538}
Z.~Chen, X.~Xiang, A source transfer domain decomposition method for
  {Helmholtz} equations in unbounded domain, {Part II: Extensions}, Numer.
  Math.:Theory Me. 6~(3) (2013) 538--555.

\bibitem{Dohrmann03:246}
C.~R. Dohrmann, \href{http://dx.doi.org/10.1137/S1064827502412887}{A
  preconditioner for substructuring based on constrained energy minimization},
  SIAM J. Sci. Comput. 25~(1) (2003) 246--258.

\bibitem{MandelDohrmann03:639}
J.~Mandel, C.~R. Dohrmann, Convergence of a balancing domain decomposition by
  constraints and energy minimization, Numer. Linear Algebra Appl. 10 (2003)
  639--659.

\bibitem{GippertKlawonn12:2208}
S.~Gippert, A.~Klawonn, O.~Rheinbach, Analysis of {FETI-DP} and {BDDC} for
  linear elasticity in {3D} with almost incompressible components and varying
  coefficients inside subdomains, SIAM J. Numer. Anal. 50~(5) (2012)
  2208--2236.

\bibitem{PavarinoWidlund10:3604}
L.~F. Pavarino, O.~B. Widlund, S.~Zampini, {BDDC} preconditioners for spectral
  element discretizations of almost incompressible elasticity in three
  dimensions, SIAM J. Sci. Comput. 32~(6) (2010) 3604--3626.

\bibitem{LiTu09:745}
J.~Li, X.~Tu, Convergence analysis of a balancing domain decomposition method
  for solving a class of indefinite linear systems, Numer. Linear Algebra Appl.
  16~(9) (2009) 745--773.

\bibitem{TuLi09:75}
X.~Tu, J.~Li, {BDDC} for nonsymmetric positive definite and symmetric
  indefinite problems, in: Domain Decomposition Methods in Science and
  Engineering XVIII, Springer, 2009, pp. 75--86.

\bibitem{SistekSousedik11:429}
J.~{\v{S}}{\'\i}stek, B.~Soused{\'\i}k, P.~Burda, J.~Mandel, J.~Novotn{\`y},
  Application of the parallel {BDDC} preconditioner to the {Stokes} flow,
  Comput. Fluids 46~(1) (2011) 429--435.

\bibitem{TuXM07:146}
X.~Tu, A {BDDC} algorithm for flow in porous media with a hybrid finite element
  discretization, Electron. Trans. Numer. Anal. 26 (2007) 146--160.

\bibitem{ZampiniTu17:A1389}
S.~Zampini, X.~Tu, Multilevel balancing domain decomposition by constraints
  deluxe algorithms with adaptive coarse spaces for flow in porous media, SIAM
  J. Sci. Comput. 39~(4) (2017) A1389--A1415.

\bibitem{DaPavarino14:A1118}
L.~B. Da~Veiga, L.~F. Pavarino, S.~Scacchi, O.~B. Widlund, S.~Zampini,
  Isogeometric {BDDC} preconditioners with deluxe scaling, SIAM J. Sci. Comput.
  36~(3) (2014) A1118--A1139.

\bibitem{KimChung15:571}
H.~H. Kim, E.~T. Chung, A {BDDC} algorithm with enriched coarse spaces for
  two-dimensional elliptic problems with oscillatory and high contrast
  coefficients, Multiscale Model. Simul. 13~(2) (2015) 571--593.

\bibitem{MandelSousedik07:1389}
J.~Mandel, B.~Soused{\'\i}k, Adaptive selection of face coarse degrees of
  freedom in the {BDDC} and the {FETI-DP} iterative substructuring methods,
  Comput. Methods Appl. Mech. Eng. 196~(8) (2007) 1389--1399.

\bibitem{KimChung18:64}
H.~H. Kim, E.~Chung, J.~Wang, {BDDC} and {FETI-DP} algorithms with a change of
  basis formulation on adaptive primal constraints, Electron. Trans. Numer.
  Anal. 49 (2018) 64--80.

\bibitem{KlawonnRadtke16:301}
A.~Klawonn, P.~Radtke, O.~Rheinbach, Adaptive coarse spaces for {BDDC} with a
  transformation of basis, in: Domain Decomposition Methods in Science and
  Engineering XXII, Springer, 2016, pp. 301--309.

\bibitem{OhWidlund18:659}
D.~S. Oh, O.~Widlund, S.~Zampini, C.~Dohrmann, {BDDC} algorithms with deluxe
  scaling and adaptive selection of primal constraints for {Raviart-Thomas}
  vector fields, Math. Comput. 87~(310) (2018) 659--692.

\bibitem{ZampiniVassilevski17:103}
S.~Zampini, P.~Vassilevski, V.~Dobrev, T.~Kolev, Balancing domain decomposition
  by constraints algorithms for {Curl}-conforming spaces of arbitrary order,
  in: P.~E. Bj{\o}rstad, S.~C. Brenner, L.~Halpern, H.~H. Kim, R.~Kornhuber,
  T.~Rahman, O.~B. Widlund (Eds.), {Domain Decomposition Methods in Science and
  Engineering XXIV}, Springer International Publishing, Cham, 2018, pp.
  103--116.

\bibitem{PengShu18:185}
J.~Peng, S.~Shu, J.~Wang, An adaptive {BDDC} algorithm in variational form for
  mortar discretizations, J. Comput. Appl. Math. 335 (2018) 185--206.

\bibitem{KimChung17:599}
H.~H. Kim, E.~T. Chung, C.~Xu, A {BDDC} algorithm with adaptive primal
  constraints for staggered discontinuous {Galerkin} approximation of elliptic
  problems with highly oscillating coefficients, J. Comput. Appl. Math. 311
  (2017) 599--617.

\bibitem{DaVeigaPavarino17:A281}
L.~B. Da~Veiga, L.~F. Pavarino, S.~Scacchi, O.~B. Widlund, S.~Zampini, Adaptive
  selection of primal constraints for isogeometric {BDDC} deluxe
  preconditioners, SIAM J. Sci. Comput. 39~(1) (2017) A281--A302.

\bibitem{BrennerSung07:1429}
S.~C. Brenner, L.-Y. Sung, {BDDC} and {FETI-DP} without matrices or vectors,
  Comput. Methods Appl. Mech. Eng. 196~(8) (2007) 1429--1435.

\bibitem{PengWang18:683}
J.~Peng, J.~Wang, S.~Shu, Adaptive {BDDC} algorithms for the system arising
  from plane wave discretization of {Helmholtz} equations, Int. J. Numer.
  Methods Eng. 116 (2018) 683--707.

\bibitem{Zampini16:S282}
S.~Zampini, {PCBDDC}: a class of robust dual-primal methods in {PETSc}, SIAM J.
  Sci. Comput. 38~(5) (2016) S282--S306.

\bibitem{KlawonnRadtke16:75}
A.~Klawonn, P.~Radtke, O.~Rheinbach, A comparison of adaptive coarse spaces for
  iterative substructuring in two dimensions, Electron. Trans. Numer. Anal. 45
  (2016) 75--106.

\bibitem{HuLi17:1242}
Q.~Hu, X.~Li, Efficient multilevel preconditioners for three-dimensional plane
  wave {Helmholtz} systems with large wave numbers, Multiscale Model. Simul.
  15~(3) (2017) 1242--1266.

\bibitem{PechsteinDohrmann17:273}
C.~Pechstein, C.~R. Dohrmann, A unified framework for adaptive {BDDC},
  Electron. Trans. Numer. Anal. 46 (2017) 273--336.

\bibitem{AndersonDuffin69:576}
W.~N. Anderson~Jr, R.~J. Duffin, Series and parallel addition of matrices, J.
  Math. Anal. Appl. 26~(3) (1969) 576--594.

\bibitem{MandelSousedik08:55}
J.~Mandel, B.~Soused{\'\i}k, C.~R. Dohrmann, Multispace and multilevel {BDDC},
  Computing 83~(2-3) (2008) 55--85.

\bibitem{DohrmannPechstein2012}
C.~R. Dohrmann, C.~Pechstein, Constraint and weight selection algorithms for
  {BDDC}, Tech. rep., Sandia National Lab.(SNL-NM), Albuquerque, NM (United
  States) (2012).

\bibitem{RixenFarhat99:489}
D.~J. Rixen, C.~Farhat, A simple and efficient extension of a class of
  substructure based preconditioners to heterogeneous structural mechanics
  problems, Int. J. Numer. Methods Eng. 44~(4) (1999) 489--516.

\end{thebibliography}

\end{document}